\documentclass[10pt,a4paper]{article}
\usepackage[utf8]{inputenc}
\usepackage{amsmath}
\usepackage{amsthm}
\usepackage{amsfonts}
\usepackage{multirow}
\usepackage[table,xcdraw]{xcolor}
\usepackage{subcaption}
\usepackage{float}
\usepackage{hyperref}
\usepackage{showlabels}
\usepackage{tabularx}
\usepackage{amssymb}
\usepackage{graphicx}
\usepackage[left=1in,right=1in,top=2cm,bottom=2cm]{geometry}
\usepackage[shortlabels]{enumitem}
\usepackage{mathrsfs}  
\usepackage[scr=rsfs]{mathalpha}
\usepackage{mathtools}  
\usepackage{amsmath}
\usepackage{amssymb}
\usepackage{tabulary}
\usepackage{booktabs}
\usepackage{hyperref}
\newcommand{\hlabel}{\phantomsection\label}
\usepackage{longfbox}

\newtheorem{Step}{Step}

\newtheorem{Theorem}{Theorem}[section]
\newtheorem{Proposition}[Theorem]{Proposition}
\newtheorem{Remark}[Theorem]{Remark}
\newtheorem{Lemma}[Theorem]{Lemma}
\newtheorem{Corollary}[Theorem]{Corollary}
\newtheorem{Definition}[Theorem]{Definition}

\newtheorem{Algorithm}{Algorithm}
\usepackage{hyperref}
\hypersetup{colorlinks=true,urlcolor=blue}
\expandafter\let\expandafter\oldproof\csname\string\proof\endcsname
\let\oldendproof\endproof
\renewenvironment{proof}[1][\proofname]{
\oldproof[\ttfamily\scshape \bf #1.]
}{\oldendproof}
\newcommand{\set}[1]{\left\{#1\right\}}
\def\prox{\mbox{\rm Prox}}
\def\ve{\varepsilon}

\def\tilde{\widetilde}

\def\dom{{\rm dom}\,}

\def\ox{\overline{x}}

\def\Bar{\overline}

\def\ra{\rangle}
\def\la{\langle}
\def\ve{\varepsilon}
\def\ox{\bar{x}}

\def\dom{\mbox{\rm dom}\,}

\def\ph{\varphi}

\def\st{\stackrel}

\def\gg{\gamma}

\def \N{{\rm I\!N}}
\def \R{{\rm I\!R}}

\newcommand{\dotproduct}[1]{\left\langle#1\right\rangle}
\newcommand{\brac}[1]{\left(#1\right)}
\newcommand{\sbrac}[1]{\left[#1\right]}
\newcommand{\abs}[1]{\left|#1\right|}
\newcommand{\norm}[1]{\left\|#1\right\|}
\numberwithin{equation}{section}

\title{\bf Inexact Proximal Methods\\ for Weakly Convex Functions}
\author{Pham Duy Khanh\footnote{Department of Mathematics, Ho Chi Minh City University of Education, Ho Chi Minh City, Vietnam. E-mail: pdkhanh182@gmail.com}\quad Boris S. Mordukhovich\footnote{Department of Mathematics, Wayne State University, Detroit, Michigan, USA. E-mail: aa1086@wayne.edu. Research of this author was partly supported by the US National Science Foundation under grants DMS-1808978 and DMS-2204519, by the Australian Research Council under Discovery Project DP-190100555, and by Project~111 of China under grant D21024.}\quad Vo Thanh Phat\footnote{Department of Mathematics, Wayne State University, Detroit, Michigan, USA. E-mail: phatvt@wayne.edu. Research of this author was partly supported by the US National Science Foundation under grants DMS-1808978 and DMS-2204519.} \quad Dat Ba Tran\footnote{Department of Mathematics, Wayne State University, Detroit, Michigan, USA. E-mail: tranbadat@wayne.edu. Research of this author was partly supported by the US National Science Foundation under grant DMS-1808978 and DMS-2204519.}}

\begin{document}
\maketitle

\noindent
{\small{\bf Abstract}. This paper proposes and develops inexact proximal methods for finding stationary points of the sum of a smooth function and a nonsmooth weakly convex one, where an error is present in the calculation of the proximal mapping of the nonsmooth term. A general framework for finding zeros of a continuous mapping  is derived from our previous paper on this subject to establish convergence properties of the inexact proximal point method when the smooth term is vanished and of the inexact proximal gradient method when the smooth term satisfies a descent condition. The inexact proximal point method achieves global convergence with constructive convergence rates when the Moreau envelope of the objective function satisfies the Kurdyka-\L ojasiewicz (KL) property. Meanwhile, when the smooth term is twice continuously differentiable with a Lipschitz continuous gradient and a differentiable approximation of the objective function satisfies the KL property, the inexact proximal gradient method achieves the global convergence of iterates with constructive convergence rates.} \\[1ex]

\section{Introduction}\label{sec:intro}

This paper is devoted to numerical variational analysis of the following class of optimization problems:
\begin{align}\label{intr: main problem}
{\rm minimize}\quad \varphi(x):=f(x)+g(x)\quad \text{ subject to }\;x\in\R^n,
\end{align}
where $f:\R^n\rightarrow\R$ is a continuously differentiable ($\mathcal{C}^1$-smooth) function, not necessary convex, and where $g:\R^n\rightarrow\Bar\R:=(-\infty,\infty]$ is a proper lower semicontinuous (l.s.c.) function. Since $g$ can be chosen as an indicator function of a closed set, the setting of \eqref{intr: main problem} encompasses various problems problems of constrained optimization, which are studied with numerous applications. This type of problem also appear  in many practical models arising, e.g., in machine learning \cite{pelckmans05,tibshirani96}, compressed sensing \cite{raginsky10}, and image processing \cite{ayers88,ochs14}. The {\em proximal-type methods} of our primal interest here are among the most exploited tools for tackling problem \eqref{intr: main problem} when $g$ is convex. More specifically, Martinet first introduced in \cite{martinet70} the proximal point method to deal with this problem when $f\equiv 0$ while the proximal gradient method is designed in \cite{beck09,nesterov13} to solve \eqref{intr: main problem} when $f$ is $\mathcal{C}^1$-smooth with a Lipschitz continuous gradient. Each iteration of these methods requires {\em exact} calculations for the {\em proximal mapping}
\begin{align}\label{analytic sol}
\prox_{\lambda g}(x):=\mathop{\rm argmin}_{y\in\R^n}\big\{{g(y)+(2\lambda)^{-1}}
\norm{y-x}^2\big\},
\end{align}
where $\lambda>0$ is a proximal parameter chosen in advance. Some choices of $g$ allow us to find an explicit form for the proximal mapping in \eqref{analytic sol}; see e.g., \cite[Section~6.9]{beckbook} and the repository in \cite{depos}. However, in many cases important for practical applications, such a form either doesn't exist, or is very expensive to be determined computationally; see, e.g., \cite{cai10,fadili11} and the discussions therein. For these reasons, some proximal methods that only require {\em inexact} information for the proximal mapping in \eqref{analytic sol} are designed. Let us mention the following major developments in this vein:

$\bullet$ Rockafellar proposes in \cite{rockafellar76} an {\em inexact proximal point method} (IPPM) for minimizing proper, l.s.c., convex functions with justifying its global convergence. A linear convergence rate of the method is achieved when the inverse of the subgradient mapping associated with the objective function is Lipschitz continuous around the origin; see also \cite{bertsekasbook}.

$\bullet$ Schmidt et al. \cite{schmidt16} study an {\em inexact proximal gradient method} (IPGM) for \eqref{intr: main problem} when $f$ is $\mathcal{C}^1$-smooth and convex with a Lipschitz continuous gradient while $g$ is proper, l.s.c., and convex. 
A linear convergence of function values achieved by the average of iterates is established in \cite{schmidt16}. If in addition $g$ is strongly convex, then it is shown in \cite{schmidt16} that the convergence is global.

$\bullet$ Bonettini et al. \cite{bonettini20} design the {\em inexact forward-backward} (iFB) {\em algorithm} for \eqref{intr: main problem} and verify its convergence when $f$ is $\mathcal{C}^1$-smooth (possibly nonconvex) with the Lipschitz continuous gradient, when $g$ is proper, l.s.c., convex, and when $\varphi$ is bounded from below. When the boundedness of the iterative sequence is assumed, stationarity for all accumulation points of the iterative sequence is established. If in addition $f$ is real analytic and $g$ is subanalytic, then the iterative sequence has finite length and thus converges to a stationary point of $\varphi$.\vspace*{0.03in}

The main goal of this paper is to develop new {\em inexact versions} of both {\em proximal point} and {\em proximal gradient} methods in the case where $f$ in \eqref{intr: main problem} is a smooth and generally nonconvex functions, while $g$
is {\em weakly convex}. The latter notion was introduced by Nurminskii \cite{nurminskii73} as a weak counterpart of {\em strongly convex} functions in the sense of Polyak \cite{polyak}. We are not familiar with inexact versions of the proximal point method for weakly convex function, while referring the reader to the recent paper \cite{bohm21}, where an {\em exact} version of the proximal gradient method has been designed for weakly convex functions. \footnote{After completing this paper that extends the scheme and results of inexact reduced first-order methods from smooth optimization in our previous paper \cite{kmt22.1} to nonsmooth problems of  type \eqref{intr: main problem} with weakly convex functions $g$, we got familiar with preprint  \cite{ewa}, which addresses problems this type assuming that $f$ is smooth and convex. The authors of \cite{ewa} develop {\em inexact forward-backward algorithms} of the splitting type that is different of ours being based on some {\em sharpness} ideas.} As shown in \cite{wang10}, the class of weakly convex functions is the {\em largest class} to study the proximal-type methods requiring the {\em single-valuedness} of the proximal mapping. Observe also that, to the best of our knowledge, {\em no convergence rate} for the iterative sequence generated by IPGM his been established when $f$ is {\em nonconvex}. Furthermore, the assumptions on the boundedness from below of $\varphi$ and the boundedness of the iterative sequence as in \cite{bonettini20} seem to be restrictive. 

Having in mind the above discussion, we propose novel {\em inexact proximal methods} to solve \eqref{intr: main problem} when $g$ is a weakly convex function and $f$ is smooth but not necessarily convex. The main ideas for these methods come mostly from a general framework for finding zeros of continuous mappings, which is an extension of the {\em inexact reduced gradient methods} (IRGM) for finding stationary points of $\mathcal{C}^1$-smooth functions introduced in \cite{kmt22.1}. Specifically, such a framework allows us to establish the global convergence with convergence rates of IPPM. The basic convergence properties of our novel version of IPGM for 
\eqref{intr: main problem} are also obtained by employing this general framework together by employing the new notion of {\em weak $\ve$-subdifferential} for weakly convex functions and techniques for dealing with approximate proximal mappings used in \cite{schmidt16}. Our {\em major contributions} can be briefly summarized as follows:

$\bullet$ Finding {\em stationary accumulation points of IPPM} for a weakly convex function. Global convergence results with explicit convergence rates for this method are obtained under the KL property for the Moreau envelope of the objective function.

$\bullet$ Finding {\em stationary accumulation points of IPGM} for structured sums \eqref{intr: main problem} of $\mathcal{C}^1$-smooth functions $f$ satisfying a descent condition and extended-real-valued weakly convex functions $g$. Global convergence results with convergence rates are obtained for this method when the smooth part is twice continuously differentiable ($\mathcal{C}^2$-smooth) with a Lipschitz continuous gradient while an approximation of the objective function satisfies the KL property.\vspace*{0.03in}

Regarding {\em numerical aspects} in comparison with the iFB algorithm in \cite{bonettini20} when $g$ is convex, our new IPGM method exhibits some advantages in errors control. More specifically, iFB uses the decreasing errors $\varepsilon_k=k^{-p}$ with $p>2$ in the approximation procedure of \cite{bonettini20} for the proximal mapping \eqref{analytic sol}. When the number of iteration $k$ is large, the error $\varepsilon_k$ becomes too small, and thus the approximation procedure needs to run longer or even cannot stop due to the round-off error of computers. On the other hand, our method adapts the type of error introduced in \cite{kmt22.1}, which does not need to be decreased after each iteration. In minimization problems for smooth functions, this error is automatically adjusted following the magnitude of the exact gradient as illustrated numerically in \cite[Section~6]{kmt22.1}. By using this type of errors, our IPGM overcomes the numerical disadvantage of iFB mentioned above.  

The rest of the paper is organized as follows. Section~\ref{sec:2} presents some basic definitions and preliminary results used throughout the paper. The main part of our work concerning the design and convergence properties of IPPM and IPGM is given in Section~\ref{sec:5}, which is split into three subsections. This major section also provides a comparison between the algorithm descriptions and theoretical results obtained by IPGM and iFB in \cite{bonettini20}. Numerical illustrations for these methods demonstrating advantages of IPGM in errors control are presented and discussed in the concluding Section~\ref{sec:6}. The concluding Section~\ref{sec:concl} discusses some directions of our future research.\vspace*{-0.1in}

\section{Basic Definitions and Preliminaries}\label{sec:2}

In this section, we recall some basic definitions and notation needed below and present their properties that are used in what follows. All our considerations are given in $\R^n$ with the Euclidean norm $\|\cdot\|$ and inner product $\dotproduct{\cdot,\cdot}$. Denoting $\N:=\{1,2,\ldots\}$ the collections of natural numbers and taking any nonempty subset $I\subset\N$, the symbol $x^k\st{I}{\to}\ox$ means that $x^k\to\ox$ as $k\to\infty$ with $k\in I$. We recall the following relaxation of the Lipschitz continuity of the gradient $\nabla f$ employed in \cite{attouch13,kmt22.1}.

\begin{Definition}[descent condition]\rm A continuously differentiable function $f:\R^n\rightarrow\R$ is said to satisfy the {\em descent condition} with constant $L>0$ if
\begin{align}\label{descent condition}
f(x)\le f(y)+\dotproduct{\nabla f(y),x-y}+\dfrac{L}{2}\norm{x-y}^2\;\text{ for all }\;x,y\in\R^n.
\end{align}
\end{Definition}
\begin{Remark}\rm\label{weakly convex of ldescent} The descent condition \eqref{descent condition} is equivalent to the convexity of $\frac{L}{2}\norm{\cdot}^2-f(\cdot)$ by the first-order characterization of convex functions \cite[Theorem~2.1.2]{nesterovbook}. This property is strictly weaker than the Lipschitz continuity of $\nabla f$ as discussed in \cite[Section~2]{kmt22.1}. It is also shown in \cite[Lemma~2.5]{davis21} that the constant in the descent condition for $f$ can be smaller than the constant of the Lipschitz continuity for $\nabla f$ if $f$ is a Moreau envelope \eqref{Moreau} of a weakly convex function (Definition~\ref{weakly convex def}). Relationships between the two properties  for $\mathcal{C}^2$-smooth functions are given below.
\end{Remark}

\begin{Proposition}\label{prop characterize l descent} Let $f:\R^n\rightarrow\R$ be a $\mathcal{C}^2$-smooth function, and let $L>0$. Then the following assertions are equivalent:
\begin{itemize}
\item[\bf(i)] $\norm{\nabla f(x)-\nabla f(y)}\le L\norm{x-y}$ for all $x,y\in\R^n$.

\item[\bf (ii)] $\abs{f(x)- f(y)-\dotproduct{\nabla f(y),x-y}}\le \dfrac{L}{2}\norm{x-y}^2$ for all $x,y\in\R^n.$

\item[\bf (iii)] $\norm{\nabla^2 f(x)}\le L$ for all $x\in\R^n$.
\end{itemize}
\end{Proposition}

\begin{proof} The proof can be deduced from  \cite[Lemmas~1.2.2 and 1.2.3]{nesterovbook} and \cite[Lemma~8 of Ch.~5]{sidfordlecture}.
\end{proof}

Next we recall the subdifferential constructions for proper extended-real-valued functions $\varphi:\R^n\rightarrow\Bar\R$ taken from \cite{mordukhovich06,mordukhovich18,rockafellarbook} and used below. The (Fr\'echet) \textit{regular subdifferential} of $\varphi$ at $\ox \in\dom \varphi$ is 
\begin{align}\label{frechet}
\widehat{\partial}\varphi(\bar x):=\Big\{v\in \R^n\;\Big|\;\liminf_{x\rightarrow\bar x}\dfrac{\varphi(x)-\varphi(\bar x)-\dotproduct{v,x-\bar x}}{\norm{x-\bar x}}\ge 0\Big\},
\end{align}
while the (Mordukhovich) \textit{limiting subdifferential} of $\varphi$ at $\bar x\in\dom \varphi$ is defined by 
\begin{align}\label{limiting subdiff}
\partial \varphi(\bar x):=\set{v\in\R^n\;\big|\;\exists\,x^k\rightarrow \ox,\ \varphi(x^k)\rightarrow\varphi(\bar x),\ \widehat{\partial}f(x^k)\ni v^k\rightarrow v}.
\end{align}
When $\varphi$ is of class of $\mathcal{C}^1$-smooth functions, then both subdifferentials in \eqref{frechet} and \eqref{limiting subdiff} reduce to the gradient $\nabla \varphi(\bar x)$. A necessary (but not sufficient) condition for $\bar x\in\R^n$ to be a local minimizer of $\varphi$ is 
\begin{align}\label{gen fermat rule}
0\in \partial\varphi(\bar x).
\end{align}
A point that satisfies \eqref{gen fermat rule} is known in the literature as an {\em M$($ordukhovich$)$-stationary} (or {\em limiting-stationary}) point of $\ph$. Since this paper does not deal with any other stationary concepts, we simply use the {\em stationary} point term in what follows. 

Let $g:\R^n\rightarrow\Bar\R$ be a proper l.s.c.\ function. Given  a parameter value $\lambda >0$, the {\it Moreau envelope} $e_{\lambda}g (\cdot):\R^n\rightarrow\R$ and the {\it proximal mapping} $\textit{\rm Prox}_{\lambda g }(\cdot):\R^n\rightrightarrows\R^n$ are defined, respectively, as
\begin{equation}\label{Moreau}
e_{\lambda}g (x):=\inf_{y \in \R^n}\Phi_{\lambda,x}(y),
\end{equation}
\begin{equation}\label{Prox} 
\prox_{\lambda g }(x):=\mathop{\rm argmin}_{y\in \R^n}\Phi_{\lambda,x}(y),
\end{equation}
where the minimizing function $\Phi_{\lambda,x}:\R^n\rightarrow\Bar\R$ is given by
\begin{align}\label{phi lambda}
\Phi_{\lambda,x}(y):=g (y)+(2\lambda)^{-1}\norm{y-x}^2.
\end{align}
A function $g:\R^n\rightarrow \Bar\R$ is called {\em prox-bounded} if there exists $\lambda>0$ such that $e_\lambda g(x)>-\infty$ for some $x\in\R^n.$ The supremum of all such $\lambda$ is the {\em threshhold $\lambda_g$ of prox-boundedness} for $g$.
When $g$ is a convex function, it is prox-bounded with the threshhold $\lambda_g=\infty$. Moreover, its proximal mapping $\prox_{\lambda g}$ is single-valued and continuous in this  case, while its Moreau envelope $e_\lambda g$ is convex and $\mathcal{C}^1$-smooth for any $\lambda>0$; see, e.g., \cite[Theorem~2.26]{rockafellarbook}. These remarkable properties can be extended for the class of {\em weakly convex} functions introduced by Nurminskii \cite{nurminskii73} in the way  equivalent to the following definition. 

\begin{Definition}\rm\label{weakly convex def}
A proper l.s.c.\ function $g:\R^n\rightarrow \Bar\R$ is called {\em $\varrho$-weakly convex} for some $\varrho\ge 0$ if the quadratically shifted function $g(\cdot)+\dfrac{\varrho}{2}\norm{\cdot}^2$ is convex.
\end{Definition}

It has been well recognized in the literature that the class of weakly convex functions is rather large including any ${\cal C}^2$-smooth function as well as many important settings in nonsmooth  and nonconvex constrained optimization; see, e.g., \cite{bohm21,wang10} and the references therein.

\begin{Remark}\rm\quad\label{remark weakly descent}

{\bf(i)} If $g$ is $\varrho$-weakly convex with $\varrho\ge0$, then it is also $\rho$-weakly convex for all $\rho\ge \varrho.$ As a consequence, any convex function is $\varrho$-weakly convex for any $\varrho\ge 0$.

{\bf(ii)} $g$ is $\mathcal{C}^1$-smooth, it follows from Remark~\ref{weakly convex of ldescent} that its $\varrho$-weak convexity is equivalent to the descent property with constant $\varrho$ of $-g$.

{\bf(iii)} Suppose that $g:\R^n\rightarrow\Bar\R$ is a $\varrho$-weakly convex function and that $\lambda\in(0,\varrho^{-1})$. Then for  any $x\in\R^n$, the minimizing function $\Phi_{\lambda,x}(y)$ in \eqref{phi lambda} is represented by
\begin{align*}
\Phi_{\lambda,x}(y)=g(y)+\dfrac{\varrho}{2}\norm{y-x}^2+\dfrac{\lambda^{-1}-\varrho}{2}\norm{y-x}^2,
\end{align*}
being {\em strongly convex} with constant $\lambda^{-1}-\varrho>0$ since $g(y)+\dfrac{\varrho}{2}\norm{y-x}^2$ is convex. 
 
{\bf(iv)} If $g$ is weakly convex, the regular and limiting subdifferentials agree for all $x\in\R^n$, i.e., $g$ is {\em lower regular} on $\R^n$  in the sense of \cite{mordukhovich06}.
\end{Remark}

The next result taken from \cite[Proposition~3.13]{wang10} shows that the {\em weak convexity} of a function is a {\em necessary and sufficient} condition for the associated {\em proximal mapping} \eqref{analytic sol} to be {\em single-valued}.

\begin{Proposition}\label{prox lipschitz}
Let $g:\R^n\rightarrow\Bar\R$ be proper, l.s.c., and prox-bounded with the threshold $\lambda_g\in (0,\infty]$. Then the following conditions are equivalent:

{\bf(i)} For all $\lambda\in (0,\lambda_g)$, the mapping $\prox_{\lambda g}$ is single-valued on $\R^n$. 

{\bf(ii)} $g$ is $\lambda_g^{-1}$-weakly convex, where we use the convention $\infty^{-1}=0$.\\
In this case, the mapping $\prox_{\lambda g}$ is Lipschitz continuous on $\R^n$  for all $\lambda\in (0,\lambda_g)$.
\end{Proposition}

Let us now recall the notion of {\em $\varepsilon$-subdifferential} broadly used in convex analysis, optimization, and their numerous applications; see, e.g., \cite{lemarechalbook,nam}.

\begin{Definition}\rm\label{epsilon subdif}
Let $g:\R^n\rightarrow\overline{\R}$ be a proper l.s.c.\ function, and let $\varepsilon\ge 0$. Then the {\em $\varepsilon$-subdifferential} of $g$ at some $\bar x\in \dom g$ is denoted by
\begin{align}\label{e-sub}
{\partial}_\varepsilon g(\bar x):= \set{v\in\mathbb{R}^n\;\big|\;\dotproduct{v,x-\bar x}\le g(x)-g(\bar x)+\varepsilon\;\mbox{ whenever }\; x\in\R^n}.
\end{align}
\end{Definition}

For convex quadratic functions important for proximal-type considerations, the $\ve$-subdifferential is explicitly calculated from the definition; see \cite[Example~1.2.2]{lemarechalbook}.

\begin{Proposition}\label{sub grad quadratic}
Let $A$ be any $n\times n$ symmetric positive-semidefinite matrix, let $b\in\R^n$, and let $g:\R^n\rightarrow\R$ be a convex quadratic function defined by $g(x):=\dfrac{1}{2}\dotproduct{Ax,x}+\dotproduct{b,x}$ for all $x\in\R^n$. Then 
\begin{align*}
\partial_\varepsilon g(x)=\Big\{A(x+e)+b\;\Big|\;\dfrac{1}{2}\dotproduct{Ae,e}\le \varepsilon\Big\}\;\text{ for all }\;x\in\R^n,\;\varepsilon\ge 0.
\end{align*}
\end{Proposition}

A direct consequence of Proposition~\ref{sub grad quadratic} is the following formula, which plays an important role in convergence analysis of inexact proximal gradient methods for convex functions conducted in \cite{schmidt16}.

\begin{Corollary}\label{sub grad norm sq}
Given $\bar x\in\R^n$, $\lambda>0$, and $g(x):=(2\lambda)^{-1}\norm{x-\bar x}^2$, we have
\begin{equation*}
\partial_\varepsilon g(x)=\big\{\lambda^{-1}(x-\bar x+e)\;\big|\;\|e\|^2\le 2\lambda \varepsilon\big\}\;\text{ for all }\;x\in\R^n,\;\varepsilon\ge 0.
\end{equation*}
\end{Corollary}
\begin{proof}
Apply Proposition~\ref{sub grad quadratic} with $A:=I/\lambda$ and $b:=-\bar x/\lambda$.
\end{proof}

Observe that the $\ve$-subdifferential \eqref{e-sub} may be {\em empty} for simple weakly convex functions. Indeed, the function $g:\R\rightarrow\R$ given by $g(x):=-x^2$ on $\R$ is obviously $2-$weakly convex while \begin{align*}
\partial_\varepsilon g(0)=\set{v\in\R\;\big|\;vx\le -x^2+\varepsilon\;\mbox{ whenever }\;x\in\R}=\emptyset\;\text{ for all }\;\varepsilon\ge 0.
\end{align*}

For our purposes in this paper, we propose to the following proximal-type modification of $\ve$-subgradients that allows us to deal with $\varrho$-weakly convex functions.

\begin{Definition}\label{approximate sub}
{\rm Let $\varrho$ and $\varepsilon$ be the nonnegative numbers. The {\em weak $\varepsilon$-subdifferential} of a $\varrho$-weakly convex function $g:\R^n\rightarrow\Bar\R$ at $\bar x\in\dom g$ is defined by 
\begin{equation}\label{weak-sub}
\partial_{\varepsilon,\varrho}g(\bar x):=\big\{v\in\mathbb{R}^n\;\big|\;\la v,x-\bar x\ra\le g(x)-g(\bar x)+\dfrac{\varrho}{2}\|x-\bar x\|^2+\varepsilon\;\text{ for all }\;x\in\R^n\big\}.
\end{equation}}
\end{Definition}

\begin{Remark}\label{prop subset}
{\rm It is straightforward to deduce from the definition that:

{\bf(i)} If $\varrho=0,$ i.e., $g$ is a convex function, then 
\begin{align*}
\partial_{\varepsilon,0}g(\bar x)=\partial_\varepsilon g(\bar x).
\end{align*}

{\bf(ii)} We have the equalities
\begin{align*}
&\partial_{\varepsilon,\varrho}g(\bar x)=\bigcap_{\delta> 0}\partial_{\varepsilon,\varrho+\delta}g(\bar x)=\bigcap_{\delta> 0}\partial_{\varepsilon+\delta,\varrho}g(\bar x).
\end{align*}

{\bf(iii)} $0\in \partial_{\varepsilon,\varrho}g(\bar x)$ if and only if $\bar x$ is an 
{\em $\varepsilon$-solution} for minimizing the convex function $\tilde{g}(\cdot):=g(\cdot)+
\dfrac{\varrho}{2}\norm{\cdot-\bar x}^2$. The latter means that
\begin{align*}
\tilde{g}(\bar x)\le\tilde{g}(x)+\varepsilon\;\mbox{ for all }\;x\in\R^n.
\end{align*}}
\end{Remark}

The next proposition represents the $\varepsilon-$weak subdifferential of a $\varrho$-weakly convex function $g$ via the $\varepsilon$-subdifferential of the quadratic shifted function $g(\cdot)+\dfrac{\varrho}{2}\norm{\cdot}^2$ and provides a sum rule for this new type of approximate subdifferentials.

\begin{Proposition}\label{prop 2.2}
Let $\varrho,\rho,\varepsilon$ be nonnegative numbers with $\nu:=\varrho+\rho$, and let $g,h:\R^n\rightarrow\Bar\R$ be $\varrho$-weakly convex and $\rho$-weakly convex, respectively. Then the following assertions hold:

{\bf(i)} For any $\bar x\in\dom g$, we have the equality
\begin{align*}
{\partial}_{\varepsilon,\varrho} g(\bar x)=\partial_{\varepsilon}(g+\dfrac{\varrho}{2}\norm{\cdot}^2)(\bar x)-\varrho\bar x.
\end{align*}

{\bf(ii)} For any $\bar x\in \dom g\cap\dom h$, it holds
\begin{align}\label{sum cup}
\partial_{\varepsilon,\nu}(g+h)(\bar x)\supset\bigcup \big\{\partial_{\varepsilon_1,\varrho}g(\bar x)+\partial_{\varepsilon_2,\rho}h(\bar x)\;\big|\;\varepsilon_i\ge 0,\;\varepsilon_1+\varepsilon_2\le \varepsilon\big\}.
\end{align}
When ${\rm ri}(\dom g)\cap{\rm ri}(\dom h)\ne\emptyset$, the equality in \eqref{sum cup} is satisfied, and as a consequence we get
\begin{align}\label{sum rule eps}
\partial_{\varepsilon,\nu} (g+h)(\bar x)\subset \partial_{\varepsilon,\varrho} g(\bar x)+\partial_{\varepsilon,\rho} h(\bar x).
\end{align}
\end{Proposition}

\begin{proof}
(i) Since $\varphi(\cdot):=g(\cdot)+\dfrac{\varrho}{2}\norm{\cdot}^2$, is convex, it follows that
\begin{align*}
\partial_{\varepsilon,\varrho} g(\bar x)&= \set{v\in\mathbb{R}^n\;\big|\;\dotproduct{v,x-\bar x}\le g(x)-g(\bar x)+\dfrac{\varrho}{2}\norm{x-\bar x}^2+\varepsilon,\ \forall x\in\R^n}\\
&=\set{v\in\mathbb{R}^n\;\big|\;\dotproduct{v+\varrho \bar x,x-\bar x}\le g(x)+\dfrac{\varrho}{2}\norm{x}^2-g(\bar x)-\dfrac{\varrho}{2}\norm{\bar x}^2+\varepsilon, \ \forall x\in\R^n}\\
&=\set{v\in\mathbb{R}^n\;\big|\;\dotproduct{v+\varrho \bar x,x-\bar x}\le \varphi(x)-\varphi(\bar x)+\varepsilon, \ \forall x\in\R^n}\\
&=\partial_{\varepsilon}\varphi(\bar x)-\varrho\bar x\;\mbox{ whenever }\;\bar x\in\dom g.
\end{align*}
(ii) We deduce from the $\varrho$-weak convexity of $g$ and $\varrho$-weak convexity of $h$ that $g(\cdot)+\dfrac{\varrho}{2}\norm{\cdot}^2$ and $h(\cdot)+\dfrac{\rho}{2}\norm{\cdot}^2$ are convex functions. Using the sum rule for $\varepsilon$-subdifferentials from \cite[Theorem~3.1.1]{lemarechalbook} gives us 
\begin{align}\label{sum qa}
\partial_\varepsilon(g+h+\dfrac{\nu}{2}\norm{\cdot}^2)(\bar x)&=\partial_\varepsilon(g+\dfrac{\varrho}{2}\norm{\cdot}^2+h+\dfrac{\rho}{2}\norm{\cdot}^2)(\bar x)\nonumber\\
 &\supset\bigcup\set{\partial_{\varepsilon_1} (g+\dfrac{\varrho}{2}\norm{\cdot}^2)(\bar x)+\partial_{\varepsilon_2} (h+\dfrac{\rho}{2}\norm{\cdot}^2)(\bar x)\;\big|\;\varepsilon_i\ge 0,\;\varepsilon_1+\varepsilon_2\le \varepsilon},
\end{align}
where the equality holds if ${\rm ri}(\dom (g+\frac{\varrho}{2}\norm{\cdot}^2))\cap{\rm ri}(\dom (h+\frac{\rho}{2}\norm{\cdot}^2))\ne\emptyset$, or equivalently if ${\rm ri}(\dom g) \cap{\rm ri}(\dom h)\ne\emptyset$. Observe by (i) that 
\begin{align*}
\partial_\varepsilon(g+h+\dfrac{\nu}{2}\norm{\cdot}^2)(\bar x)=\partial_\varepsilon(g+h)(\bar x)+\nu \bar x,
\end{align*}
and that for all $\varepsilon_1,\varepsilon_2\ge 0$ with $\varepsilon_1+\varepsilon_2\le \varepsilon$ we have
\begin{align*}
\partial_{\varepsilon_1} (g+\dfrac{\varrho}{2}\norm{\cdot}^2)(\bar x)+\partial_{\varepsilon_2} (h+\dfrac{\rho}{2}\norm{\cdot}^2)(\bar x)&=\partial_{\varepsilon_1,\varrho} g(\bar x)+\varrho\bar x+\partial_{\varepsilon_2,\rho} h(\bar x)+\rho \bar x\\
&=\partial_{\varepsilon_1,\varrho} g(\bar x)+\partial_{\varepsilon_2,\rho} h(\bar x)+\nu \bar x.
\end{align*}
Then the inclusion in \eqref{sum qa} implies that
\begin{align}\label{sum 3.7}
\partial_{\varepsilon,\nu}(g+h)(\bar x)\supset \bigcup\set{\partial_{\varepsilon_1,\varrho} (g)(\bar x)+\partial_{\varepsilon_2,\rho} (h)(\bar x)\;\big|\;\varepsilon_i\ge 0,\;\varepsilon_1+\varepsilon_2\le \varepsilon},
\end{align}
where the equality holds if ${\rm ri}(\dom g)\cap{\rm ri}(\dom h)\ne\emptyset$. When the latter condition is satisfied, we deduce from Remark~\ref{prop subset}(iii) that
\begin{align*}
\partial_{\varepsilon,\nu}(g+h)(\bar x)&=\bigcup\set{\partial_{\varepsilon_1,\varrho} g(\bar x)+\partial_{\varepsilon_2,\rho} h(\bar x)\;\big|\;\varepsilon_i\ge 0,\;\varepsilon_1+\varepsilon_2\le \varepsilon}\\
&\subset\partial_{\varepsilon,\varrho} g(\bar x)+\partial_{\varepsilon,\rho} h(\bar x),
\end{align*}
which completes the proof of the proposition.
\end{proof}

In the remainder of this section, we recall and discuss the 
{\em Kurdyka-\L ojasiewciz} (KL) {\em property} as formulated in \cite{attouch10}, which plays an important in establishing global convergence with convergence rates of many optimization methods; see, e.g., \cite{absil05,attouch13,attouch10} among other references.

\begin{Definition}\rm \label{KL ine}\rm We say that a function $f:\R^n\rightarrow\overline{\R}$ enjoys the \textit{KL property} at $\bar x\in\dom \partial f$ if there exist $\eta\in (0,\infty]$, a neighborhood $U$ of $\bar x$. and a concave continuous function $\psi:[0,\eta)\rightarrow[0,\infty)$ such that:

{\bf(i)} $\psi(0)=0$.

{\bf(ii)} $\psi$ is $\mathcal{C}^1$-smooth on $(0,\eta)$.

{\bf(iii)} $\psi'>0$ on $(0,\eta)$.

{\bf(iv)} For all $x\in U$ with $f(\bar{x})<f(x)<f(\bar{x})+\eta$, we have 
\begin{align}\label{KL 2}
\psi'\big(f(x)-f(\bar{x})\big){\rm dist}(0,\partial f(x))\ge 1.
\end{align}
\end{Definition}

It has been realized that the KL property is satisfied in broad settings. In particular, it holds at every {\em nonstationary point} of $f$; see \cite[Lemma~2.1~and~Remark~3.2(b)]{attouch10}. Furthermore, it is proved in the original paper by \L ojasiewicz \cite{lojasiewicz65} that any analytic function $f:\R^n\rightarrow\R$ satisfies the KL property at every point $\ox$ with $\psi(t)=~Mt^{1-q}$ for some $q\in[0,1)$. This property also holds for continuous 
semialgebraic functions, continuous subanalytic functions, and functions definable in an o-minimal structure  \cite{bolte07lo,bolte07}. The notions of subanalytic sets and sub-analytic functions can be found in \cite{bolte07,pangbook2} while the discussion on semialgebraic functions can be found in \cite{attouch13,benedetti90}. For the reader convenience, we recall some fundamental properties of subanalytic and semialgebraic functions in the next two remarks.

\begin{Remark}\rm\label{remark subana}
Some basic properties of subanalytic functions and subanalytic sets in \cite[p.\ 597]{pangbook2} and \cite[p.\ 1208]{bolte07} are summarized as follows:

{\bf(i)} The class of subanalytic sets is closed under finite unions and intersections.
 
{\bf(ii)} The Cartesian product of subanalytic sets is subanalytic.  

{\bf(iii)} If $f:\R^n\rightarrow\R^p$ and $g:\R^p\rightarrow\R^m$ are subanalytic mappings, with $f$ being continuous, then the composition $g\circ f$ is subanalytic. In particular, the class of continuous subanalytic functions is closed under algebraic operations.

{\bf(iv)} Analytic functions are subanalytic.

{\bf(v)} The inverse image of a subanalytic set under a subanalytic mapping is a subanalytic set.\\[0.5ex]
It follows from  (i), (ii), and (v) that if the functions $f_i:\R^n\rightarrow\R$ for $i=1,\ldots,m$ are subanalytic, then the mapping $F:=(f_1,\ldots,f_m)$ is subanalytic as well.
\end{Remark}

 \begin{Remark}\rm\label{semi algebraic}
 The following properties of semialgebraic functions are useful in the subsequent sections:
 
 {\bf(i)} Finite sums and products of semialgebraic functions are semialgebraic.
 
 {\bf(ii)} Composition of semialgebraic functions or mappings are semialgebraic.
 
 {\bf(iii)} Functions of the type $f(x):=\inf_{y\in S}h(x,y)$, where $h$ and $S$ are semialgebraic, are semialgebraic.
 
{\bf(iv)} Semialgebraic functions are subanalytic.
  
 {\bf(v)} Both regular subdifferential \eqref{frechet} and limiting subdifferential \eqref{limiting subdiff} of semialgebraic functions are 
 semialgebraic set-valued mappings.\\[0.5ex]
The properties listed in (i)-(iii) can be found in \cite{attouch10,benedetti90}. Property (iv) is mentioned in \cite[p.~1208]{bolte07lo} while the proof of (v) follows from \cite[Proposition~3.1]{ioffe09}. As a consequence of (i) and (iii), the Moreau envelope of a proper, l.s.c., semialgebraic function is also semialgebraic, and thus it is subanalytic.
\end{Remark}\vspace*{-0.2in}
 
\section{Inexact Proximal Methods for Weakly Convex Functions}\label{sec:5}

In this main section of the paper, we design and justify appropriate versions of the {\em inexact proximal point method} and the {\em inexact proximal gradient method} for the class of structured optimization problems of type \eqref{intr: main problem} with smooth nonconvex functions $f$ and extended-real-valued {\em weakly convex} functions $g$. The section is split into three subsections. Firstly, we consider a general framework for {\em finding zeros} of a single-valued {\em continuous} mapping $G:\R^n\rightarrow\R^n$ when only {\em inexact} information about $G$ is accessible. This scheme is a nonsmooth extension of our previous developments in \cite[Algorithm~1]{kmt22.1} to find stationary points of $\mathcal{C}^1$-smooth functions.\vspace*{-0.05in}

\subsection{A General Framework for Finding Zeros of Continuous Mappings}

Here we propose the following inexact algorithm to solve the problem $G(x)=0$, where $G\colon\R^n\to\R^n$ is a single-valued continuous mapping. 

\begin{Algorithm}\hlabel{scheme continuous mapping}\rm
\setcounter{Step}{-1}
\begin{Step}[initialization]\rm \label{scheme step 0}
Select a starting point $x^1\in\R^n$, initial radii $\varepsilon_1,r_1> 0$, radius reduction factors $\mu,\theta\in (0,1)$. Set $k:=1.$
\end{Step}
\begin{Step}[inexact mapping]\rm \label{scheme step 1}
Find some $g^k$ such that 
\begin{align}\label{gk G(xk)}
\norm{g^k-G(x^k)}\le \varepsilon_k.
\end{align}
\end{Step}
\begin{Step}[radius update]\rm \label{scheme step 2}
If $\norm{g^k}\le r_k+\varepsilon_k$ then $r_{k+1}:=\mu r_k,$ $\varepsilon_{k+1}:=\theta\varepsilon_k$ and $d^k:=0$. Otherwise, set $r_{k+1}:=r_k,\;\varepsilon_{k+1}:=\varepsilon_k$, and 
\begin{align}\label{dk is projection}
d^k:=-{\rm Proj}(0,\mathbb{B}(g^k,\varepsilon_k))=-\dfrac{\norm{g^k}-\varepsilon_k}{\norm{g^k}}g^k.
\end{align}
\end{Step}
\begin{Step}[stepsize]\rm \label{scheme step 3}
Choose a stepsize $t_k>0.$
\end{Step}
\begin{Step}[iteration update]\rm \label{scheme step 4}
Set $x^{k+1}:=x^k+t_kd^k.$  Increase $k$ by $1$ and go back to Step \ref{scheme step 1}.
\end{Step}
\end{Algorithm}

Following the analysis in \cite[Section~4]{kmt22.1}, we introduce the notion of null iterations of Algorithm \ref{scheme continuous mapping} and provide some fundamental properties related to this notion.

\begin{Definition}\rm\label{nul iter}
The $k^{\rm th}$ iteration of Algorithm \ref{scheme continuous mapping} is called a {\em null iteration} if $x^{k+1}=x^k$. The set of all null iterations is denoted by
$$
\mathcal{N}:=\set{k\in\N\;\big|\;x^{k+1}=x^k}.
$$
\end{Definition}

\begin{Remark}\rm \label{null remark 2} We have the following properties of null iterations for Algorithm \ref{scheme continuous mapping}:

{\bf(i)} $k\in\mathcal{N}$ if and only if either one of the following equivalent conditions holds:
\begin{align}\label{alter condi for null}
d^k=0\Longleftrightarrow r_{k+1}=\mu r_k\Longleftrightarrow\varepsilon_{k+1}=\theta\varepsilon_k\Longleftrightarrow\norm{g^k}\le r_k+\varepsilon_k.
\end{align}

{\bf(ii)} $\varepsilon_k\downarrow 0$ if and only if $r_k\downarrow 0$, which is equivalent to the set $\mathcal{N}$ being infinite.

{\bf(iii)} Considering $k\notin\mathcal{N}$, it follows from $\norm{g^k}> r_k+\varepsilon_k$ and \eqref{dk is projection} in Step \ref{scheme step 2} that 
$$\norm{d^k}=\norm{g^k}-\varepsilon_k> r_k+\varepsilon_k-\varepsilon_k=r_k.$$
As a consequence, if $\varepsilon_k\le r_k,$ we get
$\varepsilon_k\le \norm{d^k},$ and hence
\begin{align}\label{G(xk)<3dk}
\norm{G(x^k)}\le \norm{g^k}+\varepsilon_k= \norm{d^k}+2\varepsilon_k\le 3\norm{d^k}.
\end{align}

{\bf(iv)} If $\mathcal{N}$ is finite, observe that for any natural number $k\ge N:=\max{\mathcal{N}}+1$, the $k^{\rm th}$ iteration is not a null one and then deduce from (iii) that 
\begin{align*}
r_k=r_N,\;\varepsilon_k=\varepsilon_N\quad\text{and}\quad \norm{g^k}>r_k+\varepsilon_k=r_N+\varepsilon_N,\; \norm{d^k}>r_k=r_N.
\end{align*}
This tell us that $\set{g^k}$ and $\set{d^k}$ are bounded away from $0.$
\end{Remark}

The next result reveals relationships between sequences of inexact mappings, directions and radii generated by Algorithm \ref{scheme continuous mapping}. The proofs for (i) and (ii) are similar to \cite[Proposition~4.7]{kmt22.1}, while the last assertion of this proposition is deduced from (i), (ii), the continuity of $G$, and condition \eqref{gk G(xk)}.

\begin{Proposition}\label{goes to 0} Let $\set{g^k},\;\set{d^k},\;\set{\varepsilon_k}$, and $\set{r_k}$ be sequences generated by Algorithm~{\rm\ref{scheme continuous mapping}}. Then the following assertions hold:

{\bf(i)} $\varepsilon_k\downarrow0$ if and only if there is an infinite set $I\subset\N$ such that $g^k\overset{I}{\rightarrow}0$.

{\bf(ii)} For any infinite set $I\subset\N$, we have the equivalence
\begin{align*}
g^k\overset{I}{\rightarrow}0\Longleftrightarrow d^k\overset{I}{\rightarrow}0.
\end{align*}
If there exists some infinite set $I\subset\N$ such that either $g^k\overset{I}{\rightarrow}0$ or $d^k\overset{I}{\rightarrow}0$, one has  $G(x^k)\overset{I}{\rightarrow}0.$
\end{Proposition}

Now we recall the classical results that describe important properties of accumulation points generated by a sequence satisfying Ostrowski's limiting condition. Asserting (i) of Lemma~\ref{lemma: ostrowski} was first established in \cite[Theorem~28.1]{ostrowski}, while assertion (ii) of this lemma is taken from \cite[Theorems~8.3.9 and~8.3.10]{pangbook2}. 

\begin{Lemma}\label{lemma: ostrowski}
Let $\set{x^k}$ be a sequence satisfying the Ostrowski condition
\begin{align*}
\lim_{k\rightarrow\infty}\|x^{k+1}-x^k\|=0.
\end{align*}
Then the following assertions hold:

{\bf(i)} If the sequence $\set{x^k}$ is bounded, then the set of accumulation points of $\set{x^k}$ is nonempty, compact, and connected in $\R^n$.

{\bf(ii)} If $\set{x^k}$ has an isolated accumulation point, then this sequence converges to it.
\end{Lemma}

The next theorem, which is a nonsmooth extension of \cite[Theorem~5.3]{kmt22.1}, reveals major convergence properties of the sequence $\set{x^k}$ generated by 
Algorithm~\ref{scheme continuous mapping}.

\begin{Theorem}\label{corollary 2.6}
Let $\set{x^k}$ be the sequence of iterates generated by Algorithm~{\rm\ref{scheme continuous mapping}}. Assume that $\set{t_k}$ is bounded from above and $\sum_{k=1}^\infty \norm{d^k}^2<\infty$.
Then the following assertions hold:

{\bf(i)} $\varepsilon_k\downarrow0$ and $r_k\downarrow0$ as $k\rightarrow\infty.$
\item Every accumulation point of $\set{x^k}$ is a zero of $G$, and we have $\displaystyle{\sum_{k=1}^\infty}\norm{x^{k+1}-x^{k}}^2<\infty$.

{\bf(ii)} If the sequence $\set{x^k}$ is bounded, then the set of accumulation points of $\set{x^k}$ is nonempty, compact, and connected.

{\bf(iii)} If $\set{x^k}$ has an isolated accumulation point, then the entire sequence converges to this point.
\end{Theorem}

\begin{proof} Since $\sum_{k=1}^\infty \norm{d^k}^2<\infty$, we have that $d^k\rightarrow0$ as $k\to\infty$. It follows from Proposition~\ref{goes to 0} and Remark~\ref{null remark 2}(ii) that $\varepsilon_k\downarrow0$ and $r_k\downarrow0$, which in turn justifies (i) and the convergence $G(x^k)\rightarrow 0$ as $k\to\infty$. This implies that every accumulation point of $\set{x^k}$ is a zero of $G$ by taking into account the continuity of $G$. Denoting by $M$ an upper bound of $\set{t_k}$ gives us
\begin{align*}
\sum_{k=1}^\infty \norm{x^{k+1}-x^k}^2=\sum_{k=1}^\infty t_k^2\norm{d^k}^2\le M^2\sum_{k=1}^\infty\norm{d^k}^2<\infty,
\end{align*}
which implies (ii). The latter yields $\norm{x^{k+1}-x^k}\rightarrow0$ as $k\rightarrow\infty$, i.e. $\set{x^k}$ satisfies the Ostrowski condition. Then (iii) follows from Lemma~\ref{lemma: ostrowski}.
\end{proof}

Now we are going to deal with {\em non-null} iterations, which are useful to establish a global convergence of Algorithm~\ref{scheme continuous mapping} with convergence rates. If $\N\setminus\mathcal{N}$ is finite, the the iterative sequence $\set{x^k}$ 
converges finitely. Thus we consider the case where $\N\setminus\mathcal{N}$ is infinite and get the following relationships between $\set{x^k}$ and $\set{x^k}_{k\in\N\setminus\mathcal{N}}$ that are verified similarly to the proof of \cite[Proposition~4.6(iv) and (5.25)]{kmt22.1}.

\begin{Proposition}\label{0 prop non null sequence}
Let $\set{x^k}$ be the sequence generated by Algorithm~\ref{scheme continuous mapping}. Assume further that the set $\N\setminus\mathcal{N}$ is infinite and renumerate it as $\set{j_1,j_2,\ldots}$. Then defining the sequence $\set{z^k}$ by 
\begin{align}\label{0 zk xjk}
z^k:=x^{j_k}\quad \text{ for all }\;k\in\N,
\end{align}
we have that $z^{k+1}=x^{j_k+1}$, $z^k\ne z^{k+1}$ for all $k\in\N$, and that
the set of accumulation points of $\set{x^k}$ is the same as of $\set{z^k}$. Furthermore, $x^k\rightarrow\bar x$ for some $\bar x\in\R^n$ if and only if $z^k\rightarrow\bar x$ as $k\to\infty$.
\end{Proposition}

The next two propositions, taken from \cite[Theorems~5.3, 5.6, 5.9]{kmt22.1}, reveals convergence properties with convergence rates of Algorithm~\ref{scheme continuous mapping} when $G=\nabla f$, where $f:\R^n\rightarrow\R$ is a $\mathcal{C}^1$-smooth function satisfying the descent condition from Definition~\ref{descent condition}.

\begin{Proposition}\label{convergence IRG constant}
Let $\set{x^k}$ be the sequence generated by Algorithm~\ref{scheme continuous mapping} with $G=\nabla f$, where $f$ a $\mathcal{C}^1$-smooth function satisfying the descent condition for some constant $L>0$. Assume that there exist $0<\tau_1\le\tau_2<2L$ such that 
\begin{align}\label{constant}
t_k\in\sbrac{\tau_1,\tau_2}\;\text{ for all }\;k\in\N.
\end{align}
Then either $\inf f(x^k)>-\infty$, or the following assertions hold:

{\bf(i)} $\varepsilon_k\downarrow0$ and $r_k\downarrow0$ as $k\to\infty$.

{\bf(ii)} Every accumulation point of $\set{x^k}$ is a stationary point of $f$.

{\bf(iii)} If the sequence $\set{x^k}$ is bounded, then the set of accumulation points of $\set{x^k}$ is nonempty, compact, and connected.

{\bf(iv)} If $\set{x^k}$ has an isolated accumulation point, then the entire sequence $\set{x^k}$ converges to it.
\end{Proposition}

\begin{Proposition}\label{convergence IRG KL}
Let $\set{x^k}$ be the sequence generated by Algorithm~\ref{scheme continuous mapping} with $G=\nabla f$, where $f$ a $\mathcal{C}^1$-smooth function satisfying the descent condition for some constant $L>0$, and where $\set{t_k}$ is chosen as in \eqref{constant}. Assume that $f$ satisfies the KL property at some accumulation point $\bar x$ of $\set{x^k}$. Then the entire sequence $\set{x^k}$ converges to $\bar x$, which is a stationary point of $f$. Supposing in addition that $\N\setminus\mathcal{N}$ is infinite and the function given in the KL property is $\psi(t):=Mt^{1-q}$ for some $M>0$ and $q\in(0,1)$, the following convergence rates are guaranteed for the sequence $\set{z^k}$ defined as in \eqref{0 zk xjk}:

{\bf(i)} If $q\in(0,1/2]$, then the sequence $\set{z^k}$ converges linearly to $\bar x$.

{\bf(ii)} If $q\in(1/2,1)$, then there exists a positive constant $\gamma$ such that 
\begin{align*}
 \norm{z^k-\bar x}\le\gamma k^{-\frac{1-q}{2q-1}}\;\text{ for all large }\;k\in\N.
\end{align*}
\end{Proposition}\vspace*{-0.1in}

\subsection{Inexact Proximal Point Method}

Now we are ready to introduce and develop the {\em inexact proximal point method} for finding stationary points of a $\varrho$-{\em weakly convex} function $g:\R^n\rightarrow\overline\R$. The algorithm is as follows.\vspace*{0.03in}

\begin{Algorithm}[Inexact proximal point method for weakly convex functions]\label{ppm}
\setcounter{Step}{-1}

\begin{Step}[initialization]\rm
Select a proximal parameter $\lambda \in\big(0,\varrho^{-1}\big)$ and some $\tau_1,\tau_2$ satisfying the conditions $0<\tau_1\le \tau_2<2\lambda$. Choose a starting point $x^1\in \mathbb{R}^n$, initial radii $\varepsilon_1,r_1>0$, and radius reduction factors $\mu,\theta\in (0,1)$. Set $k:=1.$
\end{Step}

\begin{Step}[inexact proximal mapping]\rm\label{step 1 ppm}
Find some $p^k$ and define $g^k$ such that \begin{align}\label{inexact proximal mapping}
\norm{p^k-\prox_{\lambda g}(x^k)}\le \lambda\varepsilon_k \text{ and }g^k:=\lambda^{-1}(x^k-p^k).
\end{align}
\end{Step}

\begin{Step}[radius update]\rm
If $\norm{g^k}\le r_k+\varepsilon_k$, then $r_{k+1}:=\mu r_k,\varepsilon_{k+1}:=\theta\varepsilon_k,d^k:=0$. Otherwise, set $r_{k+1}:=r_k,\;\varepsilon_{k+1}:=\varepsilon_k$, and 
\begin{align*}
d^k:=-\dfrac{\norm{g^k}-\varepsilon_k}{\norm{g^k}}g^k.
\end{align*}
\end{Step}

\begin{Step}[stepsize]\rm
Choose a stepsize $t_k\in [\tau_1,\tau_2].$
\end{Step}

\begin{Step}[iteration update]\rm
Set $x^{k+1}:=x^k+t_k d^k $. Increase $k$ by $1$ and go back to Step \ref{step 1 ppm}.
\end{Step}
\end{Algorithm}

We recall some important properties of Moreau envelopes of weakly convex functions that are useful in the convergence analysis of Algorithm \ref{ppm}. 

\begin{Proposition}\label{weakly and descent}
Let $g :\R^n\rightarrow\overline{\R}$ be a $\varrho$-weakly convex function, and let $\lambda\in(0,\varrho^{-1})$. Then the following assertions holds:

{\bf(i)} The Moreau envelope $e_{\lambda}g$ is $\mathcal{C}^1$-smooth  and its gradient calculated by
\begin{align}\label{prop 2.8 i}
\nabla e_{\lambda}g (x)=\lambda^{-1}(x-\prox_{\lambda g }(x)).
\end{align}

{\bf(ii)} The Moreau envelope $e_{\lambda}g$ satisfies the descent condition with constant $\lambda^{-1}$ meaning that
\begin{align}
e_{\lambda}g (y)\le e_{\lambda}g (x)+\dotproduct{\nabla e_{\lambda}g(x),y-x}+(2\lambda)^{-1}\norm{y-x}^2\text{ for all }x,y\in\R^n.
\end{align}

{\bf(iii)} Every stationary point of $e_\lambda g$ is a stationary point of $g$.
\end{Proposition}

\begin{proof} Assertions (i) and (ii) follow from \cite[Lemma~2.5]{davis21}. Taking any $x$ with $\nabla e_\lambda g(x)=0$, we deduce from \eqref{prop 2.8 i} that $x=\prox_{\lambda g}(x)$. By \eqref{Prox}, $x$ solves the problem
\begin{align*}
\mathop{\rm minimize}_{y\in\R^n}\big\{g(y)+(2\lambda)^{-1}\norm{y-x}^2\big\}.
\end{align*}
Applying the generalized Fermat rule and the subdifferential sum rule to the function $y\mapsto g(y)+(2\lambda)^{-1}\norm{y-x}^2$ gives us $0\in\partial g(x)$, which verifies assertion (iii).
\end{proof}

\begin{Remark}\rm \label{reduce prop}
Proposition~\ref{weakly and descent}(i) allows us to conclude that the inexact proximal point method in Algorithm~\ref{ppm} is a special case of the general scheme to find zeros of continuous functions proposed in Algorithm~\ref{scheme continuous mapping}. Indeed, we can equivalently rewrite condition \eqref{inexact proximal mapping} as follows:
\begin{align*}
\norm{g^k-
\nabla e_\lambda g(x^k)}&=
\norm{ \lambda^{-1}(x^k-p^k)-
\lambda^{-1}(x^k-{\rm Prox}_{\lambda g}(x^k))}\\
&=\lambda^{-1}\norm{{\rm Prox}_{\lambda g}(x^k)-p^k}\le \varepsilon_k.
\end{align*}
Therefore, Algorithm~\ref{scheme continuous mapping} reduces to Algorithm~\ref{ppm} with $G:=\nabla e_\lambda g$ and $t_k\in[\tau_1,\tau_2]$.
\end{Remark}

The next theorem, which is based on Proposition~\ref{weakly and descent} and Theorem \ref{convergence IRG constant}, reveals convergence properties of the inexact proximal point Algorithm~\ref{scheme continuous mapping} to find stationary points of weakly convex functions.

\begin{Theorem}\label{convergence ppm}
Let $g:\R^n\rightarrow\Bar\R$ be a $\varrho$-weakly convex function with $\varrho>0$, and let $\set{x^k}$ be the sequence generated by Algorithm~\ref{ppm}. If $\norm{x^k}\not\rightarrow\infty$, then the following assertions hold:

{\bf(i)} $\varepsilon_k\downarrow0$ and $r_k\downarrow 0$ as $k\to\infty$.

{\bf(ii)} Every accumulation point of $\set{x^k}$ is a stationary point of $g$.

{\bf(iii)} If the sequence $\set{x^k}$ is bounded, then the set of accumulation points of $\set{x^k}$ is nonempty, compact, and connected.

{\bf(iv)} If $\set{x^k}$ has an isolated accumulation point, then the entire sequence $\set{x^k}$ converges to it.
\end{Theorem}

\begin{proof}
It follows from Remark~\ref{reduce prop} that Algorithm~\ref{scheme continuous mapping} reduces to Algorithm~\ref{ppm} with $G=\nabla e_\lambda g$ and $t_k\in [\tau_1,\tau_2]$ for all $k\in\N$. By the descent property with constant $\lambda^{-1}$ of $e_\lambda g$ from Proposition~\ref{weakly and descent}(ii) and the choices of $\tau_1,\tau_2$, the assumptions for Proposition~\ref{convergence IRG constant} with $f:=e_\lambda g$ and $L:=\lambda^{-1}$ are satisfied. This ensures the fulfillment of assertions (i), (iii), (iv) and that every accumulation point of $\set{x^k}$ is a stationary point of $e_\lambda g$. Combining this with Proposition~\ref{weakly and descent}(iii) yields (ii) and completes the proof.  
\end{proof}

The final result of this subsection employs the KL property for Moreau envelopes to establish the global convergence of the entire sequence of iterates in Algorithm~\ref{ppm} to a stationary point of a weakly convex function with giving constructive convergence rates.

\begin{Theorem}\label{convergence KL ppm}
Let $g:\R^n\rightarrow\Bar\R$ be a $\varrho$-weakly convex function with $\varrho\ge 0$, and let $\set{x^k}$ be the sequence generated by Algorithm~\ref{ppm} with $\varepsilon_1\le r_1$ and $\theta\le\mu$. Assume that $e_\lambda g$ satisfies the KL property at some accumulation point $\bar x$ of $\set{x^k}$. Then the entire sequence $\set{x^k}$ converges to $\bar x$, which is a stationary point of $g$. If $\N\setminus\mathcal{N}$ is infinite and the function $\psi$ in the KL condition is  $\psi(t):=Mt^{1-q}$ for some $M>0$ and $q\in(0,1)$, then the following convergence rates are guaranteed for the sequence $\set{z^k}$ from \eqref{0 zk xjk}:

{\bf(i)} For $q\in(0,1/2]$, the sequence $\set{z^k}$ converges linearly to $\bar x$.

{\bf(ii)} For $q\in(1/2,1)$, there exists $\gg>0$ such that 
\begin{align*}
 \norm{z^k-\bar x}\le\gamma k^{-\frac{1-q}{2q-1}}\;\text{ for all large }\;k\in\N.
\end{align*}
\end{Theorem}

\begin{proof}
By Remark~\ref{reduce prop}, Algorithm~\ref{ppm} reduces to Algorithm~\ref{scheme continuous mapping} with $G:=\nabla e_\lambda g$ and $t_k\in[\tau_1,\tau_2]$, $k\in\N$. If $\set{x^k}$ has an accumulation point $\bar x$ and $e_\lambda g$ satisfies the KL property at $\bar x$, then the assumptions of Proposition~\ref{convergence IRG KL} are satisfied for $f:=e_\lambda g$. Therefore, the claimed results hold with taking into account that stationary points of $f$ are stationary point of $g$ by Proposition~\ref{weakly and descent}(iii).
\end{proof}\vspace*{-0.2in}

\subsection{Inexact Proximal Gradient Method} 

In this subsection, we derive our major results on the new {\em inexact proximal gradient method} to solve optimization problems of the type
\begin{align}\label{eq:composite}
{\rm minimize}\quad \varphi(x):=f(x)+g(x),
\end{align}
where $f:\R^n\rightarrow\R$ is a $\mathcal{C}^1$-smooth function satisfying the descent condition \eqref{descent condition} with some $L>0$, and where $g:\R^n\rightarrow\Bar\R$ is a proper, l.s.c., $\varrho$-weakly convex function with $\varrho\ge 0$. As mention in Section~\ref{sec:intro},  the formally unconstrained framework of \eqref{eq:composite} covers problems of constrained optimization via $x\in\dom g$.

Define the {\em composite gradient} mapping $T_\lambda:\R^n\rightarrow\R^n$ as $\lambda\in(0,\varrho^{-1})$ with $0^{-1}:=\infty$ by
\begin{align*} T_\lambda (x)=\prox_{\lambda g}(F_\lambda(x)),
\end{align*}
where $F_\lambda(x):=x-\lambda \nabla f(x)$. The mapping $T_\lambda$ can be expressed by \eqref{phi lambda} as 
\begin{align}\label{T and fermat}
T_\lambda (x)=\mathop{\rm argmin}_{y\in\R^n}\Phi_{\lambda,F_\lambda(x)}(y),
\end{align}
providing the following {\em quadratic approximation} of $\varphi$ (see \cite{beck09}):
\begin{align}
T_\lambda(x)
&=\mathop{\rm argmin}_{y\in\R^n}\big\{g(y)+(2\lambda)^{-1}\norm{y-(x-\lambda\nabla f(x))}^2\big\}\nonumber\\
&=\mathop{\rm argmin}_{y\in\R^n}\big\{(\lambda/2)\norm{\nabla f(x)}^2+\dotproduct{\nabla f(x),y-x}+(2\lambda)^{-1}\norm{y-x}^2+g(y)\big\}\nonumber\\
&=\mathop{\rm argmin}_{y\in\R^n}\big\{f(x)+\dotproduct{\nabla f(x),y-x}+(2\lambda)^{-1}\norm{y-x}^2+g(y)\big\}.\label{argmin subproblem}
\end{align}
Define further the \textit{gradient} mapping $G_\lambda:\R^n\rightarrow\R^n$ by
\begin{align}\label{def G lambda}
G_\lambda(x):=\lambda^{-1}(x-T_\lambda(x)).
\end{align}

\begin{Remark}\rm\label{prop: prox property} By Proposition~\ref{prox lipschitz}, $\prox_{\lambda g}$ is continuous. Then the continuity of $\nabla f$ yields this property for $F_\lambda,T_\lambda$, and $G_{\lambda}$.
Moreover, $G_\lambda(x)=0$ implies that $x=T_\lambda(x).$ It follows from \eqref{argmin subproblem} with taking into account the generalized Fermat's rule and the subdifferential sum rule that
\begin{align*}
0\in \nabla f(x)+\partial g(x)=\partial\varphi(x).
\end{align*}
\end{Remark}

Now we are ready to propose the {\em inexact proximal gradient method} (IPGM) to find stationary points of the class of functions $\varphi$ given \eqref{eq:composite} under the assumptions therein.

\begin{Algorithm}[Inexact proximal gradient method]\rm\hlabel{pgm}
\setcounter{Step}{-1}
\begin{Step}[initialization]\rm
Select a proximal parameter $\lambda \in\big(0,(L+\varrho)^{-1}\big),$ a starting point $x^1\in \mathbb{R}^n$, initial radii $\varepsilon_1,r_1> 0$, radius reduction factors $\mu,\theta\in (0,1)$, and sequence of manually controlled errors $\set{\rho_k}\subset[0,\infty)$. Set $k:=1.$
\end{Step}
\begin{Step}[inexact proximal mapping]\rm\label{Step 1}
Find a vector $p^k\in\R^n$ satisfying
\begin{align}\label{subproblem}
\Phi_{\lambda,F_\lambda(x^k)}(p^k)\le \inf  \Phi_{\lambda,F_\lambda(x^k)}+ \min\set{\dfrac{\lambda(1-\lambda\varrho)\varepsilon_k^2}{2},\rho_k}
\end{align}
and define $g^k:=\lambda^{-1}(x^k-p^k)$.
\end{Step}
\begin{Step}[iteration update]\rm\label{Step 2}
If $\norm{g^k}\le r_k+\varepsilon_k$, then $r_{k+1}:=\mu r_k,\varepsilon_{k+1}:=\theta\varepsilon_k$ and $x^{k+1}:=x^k.$ Otherwise, set $r_{k+1}:=r_k,\varepsilon_{k+1}:=\varepsilon_k$, and $x^{k+1}:=p^k$. Increase $k$ by $1$ and go back to Step~\ref{Step 1}.
\end{Step}
\end{Algorithm}

The next proposition shows that the sequence $\set{x^k}$ in Algorithm~\ref{pgm} is the one generated by Algorithm~\ref{scheme continuous mapping} for finding zeros of the continuous mapping $G_\lambda$ from \eqref{def G lambda}.

\begin{Proposition}\label{tk < 2 lambda}
Let $\set{g^k},\set{r_k}$, and $\set{\varepsilon_k}$ be the sequences generated by Algorithm~\ref{pgm}, and let the sequences of directions $\set{d^k}$ and stepsizes $\set{t_k}$ be defined as 
\begin{align}\label{1.8}
d^k:=\begin{cases}0&\text{\rm if }\norm{g^k}\le r_k+\varepsilon_k,\\
-\dfrac{\norm{g^k}-\varepsilon_k}{\norm{g^k}}g^k&\text{\rm otherwise}
\end{cases}\quad \text{ and }\quad t_k:=\begin{cases}2\lambda&\text{\rm if }d^k=0,\\
\lambda\dfrac{\norm{g^k}}{\norm{d^k}}&\text{\rm if }d^k\ne 0.
\end{cases}
\end{align}
Then the following assertions hold for all $k\in\N:$

{\bf(i)} $\norm{p^k-T_\lambda(x^k)}\le\min \set{\lambda \varepsilon_k, \sqrt{\dfrac{2\rho_k}{\lambda^{-1}-\varrho}}}$.

{\bf(ii)}  $x^{k+1}=x^k+t_kd^k$.

{\bf(iii)}  $t_k\ge \lambda$ and if in addition $\varepsilon_k\le r_k$, then $t_k\le 2\lambda$.

{\bf(iv)} $\norm{g^k-G_\lambda(x^k)}\le \varepsilon_k$ and $\dotproduct{G_\lambda(x^k),d^k}\le -\norm{d^k}^2.$
\end{Proposition}

\begin{proof}
By Remark~\ref{remark weakly descent}(iii), the function $\Phi_{\lambda,F_\lambda(x^k)}$ is strongly convex with constant $\lambda^{-1}-\varrho$. Using \cite[Corollary~3.2.3]{nesterovbook} gives us the estimate
\begin{align*}
\Phi_{\lambda,F_\lambda(x^k)}(y)-\Phi_{\lambda,F_\lambda(x^k)}(T_\lambda(x^k))\ge \dfrac{\lambda^{-1}-\varrho}{2}\norm{y-T_\lambda(x^k)}^2\text{ for all }y\in\R^n.
\end{align*}
Then the choice of $p^k$ in \eqref{subproblem} yields
\begin{align*}
\dfrac{\lambda^{-1}-\varrho}{2}\norm{p^k-T_\lambda(x^k)}^2\le\min\set{ \dfrac{\varepsilon_k^2\lambda^2(\lambda^{-1}-\varrho)}{2},\rho_k}.
\end{align*}
Dividing both sides by $ \dfrac{\lambda^{-1}-\varrho}{2}>0$ and then taking the square root verify assertion (i).

The two assertions (ii) and (iii) obviously hold if $d^k=0$. Otherwise,  $\norm{g^k}>r_k+\varepsilon_k$ and $\norm{d^k}=\norm{g^k}-\varepsilon_k$ from \eqref{1.8}. By using $g^k=\lambda^{-1}(x^k-p^k)$ in Step~\ref{Step 1}, we get
\begin{align*}
t_kd^k=-\lambda \dfrac{\norm{g^k}}{\norm{d^k}} \cdot\dfrac{\norm{g^k}-\varepsilon_k}{\norm{g^k}}g^k=-\lambda g^k=p^k-x^k.
\end{align*}
Combining this with $x^{k+1}=p^k$ from Step~\ref{Step 2} yields $x^k+t_kd^k=p^k=x^{k+1}$, which verifies (ii). Moreover,
\begin{align*}
t_k=\lambda \dfrac{\norm{g^k}}{\norm{d^k}}=\lambda \dfrac{\norm{g^k}}{\norm{g^k}-\varepsilon_k}\ge \lambda.
\end{align*}
Now taking into account that $\varepsilon_k\le r_k$ brings us to
\begin{align*}
\norm{d^k}=\norm{g^k}-\varepsilon_k>r_k+\varepsilon_k-\varepsilon_k=r_k\ge \varepsilon_k,
\end{align*}
which in turn implies that
\begin{align*}
t_k=\lambda\dfrac{\norm{g^k}}{\norm{d^k}}=\lambda\dfrac{\norm{d^k}+\varepsilon_k}{\norm{d^k}}= \lambda\brac{1+\dfrac{\varepsilon_k}{\norm{d^k}}}\le 2\lambda
\end{align*}
and therefore justifies assertion (iii).

By (i) and the constructions of $g^k$ in Step \ref{Step 1} as well as $G_\lambda(x^k)$ in \eqref{def G lambda}, we get 
\begin{align*}
\norm{g^k-G_\lambda(x^k)}=\norm{\lambda^{-1}(x^k-p^k)-\lambda^{-1}(x^k-T_\lambda(x^k))}= \lambda^{-1}\norm{p^k-T_\lambda(x^k)}\le \varepsilon_k,
\end{align*}
which verifies the first assertion in (iv). The second assertion therein in the case where $d^k\ne 0$ follows from the characterization of projections by noting that $d^k=-{\rm Proj}(0|\mathbb{B}(g^k,\varepsilon_k))$ and $G_\lambda(x^k)\in \mathbb{B}(g^k,\varepsilon_k)$. This completes the proof of the proposition.
\end{proof}

The following result addresses the inexact calculation of proximal points in 
Step~\ref{Step 1} of Algorithm \ref{pgm}.

\begin{Proposition}\label{prop: algorithm}
For any $k\in\N$, there are $\xi_k\in \partial_{\rho_k,\varrho}g(p^k)$ and $e^k\in \R^n$ with $\norm{e^k}\le \sqrt{2\lambda\rho_k}$ and
\begin{align*}
0=\xi_k+\lambda^{-1}(p^k-F_\lambda(x^k)+e^k).
\end{align*}
\end{Proposition}
\begin{proof}
 Condition \eqref{subproblem} tells us that $p^k$ is a $\rho_k$-solution to the problem of minimizing the convex function $\Phi_{\lambda,F_\lambda(x^k)}$. It follows from \cite[Theorem~1.1.5]{lemarechalbook} that 
\begin{align}\label{0 in partial}
0&\in \partial_{\rho_k} \Phi_{\lambda,F_\lambda(x^k)}(p^k).
\end{align}
Let $h:\R^n\rightarrow\R$ be the convex quadratic function defined by $h(\cdot):=(2\lambda^{-1}\norm{\cdot-F_\lambda(x^k)}^2.$ It follows from \eqref{phi lambda} that $\Phi_{\lambda,F_\lambda(x^k)}(\cdot)=g(\cdot)+h(\cdot).$
By Remark~\ref{prop subset} and Proposition~\ref{prop 2.2}(ii), we have
\begin{align*}
&\partial_{\rho_k} \Phi_{\lambda,F_\lambda(x^k)}(p^k)\subset\partial_{\rho_k,\varrho} \Phi_{\lambda,F_\lambda(x^k)}(p^k)\\
&\subset {\partial}_{\rho_k,\varrho}g(p^k)+{\partial}_{\rho_k,0}h(p^k)={\partial}_{\rho_k,\varrho}g(p^k)+{\partial}_{\rho_k}h(p^k).
\end{align*}
Combining the latter with \eqref{0 in partial} lead us to the inclusion
\begin{align}\label{0 in separable}
0\in {\partial}_{\rho_k,\varrho}g(p^k)+{\partial}_{\rho_k}h(p^k).
\end{align}
Furthermore, it follows from Corollary~\ref{sub grad norm sq} that
\begin{align*}
{\partial}_{\rho_k}h(p^k)=\big\{\lambda^{-1}(p^k-F_\lambda(x^k)+e)\big| \norm{e}\le\sqrt{ 2\lambda \rho_k}\big\}.
\end{align*}
Then \eqref{0 in separable} yields the existence of some $\xi_k\in \partial_{\rho_k,\varrho}g(p^k)$ and $e^k\in\R^n$ with $\norm{e^k}\le \sqrt{2\lambda\rho_k}$ such that 
\begin{align*}
0=\xi_k+\lambda^{-1}(p^k-F_\lambda(x^k)+e^k),
\end{align*}
and thus we are done with the verification of this proposition.
\end{proof}

The next proposition is crucial in deriving convergence properties of 
Algorithm~\ref{pgm}.

\begin{Proposition}\label{prop dk goes 0}
At iteration $k^{\rm th}$ of Algorithm~\ref{pgm}, suppose that $\varepsilon_k\le r_k$. Then  we have the estimate
\begin{align}
C_1\norm{d^k}^2\le \varphi(x^k)-\varphi(x^{k+1})+C_2\sqrt{\rho_k}\norm{d^k}+\rho_k,
\end{align}\label{dir-est}
where the constants $C_1$ and $C_2$ are computed by
\begin{align}\label{def C1 C2}
C_1:=\lambda(1-\lambda(L+\varrho))\quad \text{ and }\quad C_2:=2\brac{\sqrt{\dfrac{2}{\lambda^{-1}-\varrho}}+\sqrt{2\lambda}}.
\end{align}
\end{Proposition}

\begin{proof}
Estimate \eqref{dir-est} obviously holds at null iterations, and so we suppose that $k\notin\mathcal{N}$. It follows from \eqref{argmin subproblem}  the composite gradient mapping representation
\begin{align*}
T_\lambda(x^k)=\mathop{\rm argmin}_{y\in\R^n}\big\{f(x^k)+\dotproduct{\nabla f(x^k),y-x^k}+(2\lambda)^{-1}\norm{y-x^k}^2+g(y)\big\}.
\end{align*}
Applying the generalized Fermat rule to this problem, we find $\xi_\lambda(x^k)\in \partial g(T_\lambda(x^k))$ such that 
\begin{align}\label{def xi k}
0=\nabla f(x^k)+\lambda^{-1}(T_\lambda(x^k)-x^k)+\xi_\lambda(x^k)=\xi_\lambda(x^k)+\lambda^{-1}(T_\lambda(x^k)-F_\lambda(x^k)).
\end{align}
By Proposition~\ref{prop: algorithm}, there exist $\xi_k\in \partial_{\rho_k,\varrho}g(p^k)$ and $e^k\in \R^n$ with $\norm{e^k}\le \sqrt{2\lambda\rho_k}$ and
\begin{align}\label{xi delta k}
 0=\xi_k+\lambda^{-1}(p^k-F_\lambda(x^k)+e^k).
\end{align}
It follows from \eqref{xi delta k}, the descent property of $f$ with constant $L$, and the first equality in \eqref{def xi k} that
\begin{align}
\varphi (x^{k+1})&\overset{\eqref{descent condition}}{=} f(x^{k+1})+g(x^{k+1})\le  f(x^k)+ t_k\dotproduct{\nabla f(x^k),d^k}+\dfrac{Lt_k^2}{2}\norm{d^k}^2+ g(x^{k+1})\nonumber\\
&\overset{\eqref{def xi k}}{=} f(x^k)+t_k\dotproduct{\dfrac{1}{\lambda}(x^k-T_\lambda(x^k))-\xi_\lambda(x^k),d^k}+\dfrac{Lt_k^2}{2}\norm{d^k}^2+ g(x^{k+1})\nonumber\\
&\overset{\eqref{def G lambda}}{=} f(x^k)+t_k\dotproduct{G_\lambda(x^k),d^k}+\dfrac{Lt_k^2}{2}\norm{d^k}^2+ g(x^{k+1})-t_k\dotproduct{\xi_\lambda(x^k),d^k}\nonumber\\
&\le  \brac{f(x^k)-t_k \norm{d^k}^2+\dfrac{Lt_k^2}{2}\norm{d^k}^2}+ \brac{g(x^{k+1})+\dotproduct{\xi_\lambda(x^k),x^k-x^{k+1}}},\label{ine 1}
\end{align}
where the last inequality is implied by Proposition~\ref{tk < 2 lambda}(ii,iv). The second term on the right hand side in the latter inequality can be expressed as
\begin{align}\label{ine 2}
g(x^{k+1})+\dotproduct{\xi_\lambda(x^k),x^k-x^{k+1}}&=g(x^{k+1})+\dotproduct{\xi_k,x^k-x^{k+1}}+\dotproduct{\xi_\lambda(x^k)-\xi_k,x^k-x^{k+1}}.
\end{align}
Since $k\notin\mathcal{N},$ Step \ref{Step 2} gives us $x^{k+1}=p^k.$
Using the definition of approximate subdifferential for weakly convex functions \eqref{approximate sub} with taking into account $\xi_k\in \partial_{\rho_k,\varrho}g(x^{k+1})$ gives us 
\begin{align}\label{ine 3}
g(x^{k+1})+\dotproduct{\xi_k,x^k-x^{k+1}}&\le  g(x^k)+\dfrac{\varrho}{2}\norm{x^k-x^{k+1}}^2+\rho_k\nonumber\\
&=g(x^k)+\dfrac{\varrho t_k^2}{2}\norm{d^k}^2+\rho_k.
\end{align}
Combining the equalities in \eqref{def xi k}, \eqref{xi delta k} with $\norm{e^k}\le \sqrt{2\lambda\rho_k}$ and Proposition~\ref{tk < 2 lambda}(i), we get
\begin{align}\label{ine 4}
&\dotproduct{\xi_\lambda(x^k)-\xi_k,x^k-x^{k+1}}=\lambda^{-1}\dotproduct{p^k+e^k-T_{\lambda}(x^k),x^k-x^{k+1}}\nonumber\\
&\le\lambda^{-1}\norm{p^k-T_\lambda(x^k)+e^k}\cdot\norm{x^k-x^{k+1}\le\lambda^{-1}t_k}\brac{\norm{p^k-T_\lambda(x^k)}+\norm{e^k}}\cdot\norm{d^k}\nonumber\\
&\le \dfrac{t_k}{\lambda}\brac{\sqrt{\dfrac{2\rho_k}{\lambda^{-1}-\varrho}}+\sqrt{2\lambda\rho_k}}\norm{d^k}.
\end{align}
It follows from \eqref{ine 2}--\eqref{ine 4} the fulfillment of the estimate
\begin{align}\label{ine 6}
g(x^{k+1})+\dotproduct{\xi_\lambda(x^k),x^k-x^{k+1}}\le g(x^k)+\dfrac{\varrho t_k^2}{2}\norm{d^k}^2+\rho_k+\dfrac{t_k}{\lambda}\brac{\sqrt{\dfrac{2\rho_k}{\lambda^{-1}-\varrho}}+\sqrt{2\lambda\rho_k} }\norm{d^k}.
\end{align}
Since $\varepsilon_k\le r_k$, Proposition~\ref{tk < 2 lambda}(iii) tells us $t_k\in [\lambda,2\lambda]$. Combining \eqref{ine 6} with \eqref{ine 1}, $t_k\le 2\lambda$, and $\varphi(x^k)=f(x^k)+g(x^k)$ ensures that
\begin{align}\label{ine 5}
\varphi(x^{k+1})&\le \varphi(x^k)-t_k\norm{d^k}^2+\dfrac{(L+\varrho)t_k^2}{2}\norm{d^k}^2+\rho_k+\dfrac{t_k}{\lambda}\brac{\sqrt{\dfrac{2\rho_k}{\lambda^{-1}-\varrho}}+\sqrt{2\lambda\rho_k} }\norm{d^k}\nonumber\\
&\le \varphi(x^k)-\brac{t_k-\dfrac{(L+\varrho)t_k^2}{2}}\norm{d^k}^2+C_2\sqrt{\rho_k}\norm{d^k}+\rho_k.
\end{align}
By using $t_k\in [\lambda,2\lambda]$, we also arrive at
\begin{align*}
\brac{t_k-\dfrac{(L+\varrho)t_k^2}{2}}= t_k\brac{1-\dfrac{(L+\varrho)t_k}{2}}\ge \lambda (1-\lambda(L+\varrho))=C_1.
\end{align*}
Finally, applying the latter to \eqref{ine 5} leads us to 
\begin{align*}
\varphi(x^{k+1})&\le\varphi(x^k)-C_1\norm{d^k}^2+C_2\sqrt{\rho_k}\norm{d^k}+\rho_k,
\end{align*}
which is claimed in the proposition.
\end{proof}

The following major result for our IPGM considers two types of errors in Algorithm~\ref{pgm} and shows that the algorithm generates stationary accumulation points and converges to isolated stationary point.

\begin{Theorem}\label{stationarity of pgm}
Let $C_1,C_2$ be defined in \eqref{def C1 C2}, and let $\set{x^k}$ be the sequence generated by Algorithm~\ref{pgm} such that $\varepsilon_1\le r_1$, $\theta\le \mu$,
and either
\begin{align}
\sum_{k=1}^\infty \rho_k<\infty,\quad\text{ or }\quad \rho_k=C\varepsilon_k^2 \quad\text{ for all }\;k\in\N,
\end{align}
where $C=\min\set{\dfrac{C_1^2}{4C_2^2},\dfrac{C_1}{4}}$. Assuming that b$\inf_{k\in\N}\varphi(x^k)>-\infty$, we have the assertions:

{\bf(i)} $\varepsilon_k,r_k\downarrow0$ as $k\to\infty$.

{\bf(ii)} Every accumulation point of $\set{x^k}$ is a stationary point of $\varphi$ and $\norm{x^k-x^{k+1}}\rightarrow0$ as $k\to\infty$.

{\bf(iii)} If the sequence $\set{x^k}$ is bounded, then the set of accumulation points of $\set{x^k}$ is nonempty, compact, and connected.

{\bf(iv)} If $\set{x^k}$ has an isolated accumulation point, then the entire sequence $\set{x^k}$ converges to it.
\end{Theorem}

\begin{proof}
It follows from Proposition~\ref{tk < 2 lambda} that $\set{x^k}$ is indeed a sequence generated by Algorithm~\ref{scheme continuous mapping} for finding zeros of the continuous mapping $G_\lambda$. By Theorem~\ref{corollary 2.6}, it suffices to show that  ${\sum_{k=1}^\infty}\norm{d^k}^2<\infty$ to get all the assertions in (i)-(iv). The conditions $\varepsilon_1\le r_1$ and $\theta\le \mu$ yield $\varepsilon_k\le r_k$, $k\in\N$. We consider the two choices for $\set{\rho_k}$. Suppose first that $\sum_{k=1}^\infty \rho_k<\infty$. It follows from Proposition \ref{prop dk goes 0} that 
\begin{align}\label{C1 C2}
C_1\norm{d^k}^2\le \varphi(x^k)-\varphi(x^{k+1})+C_2\sqrt{\rho_k}\norm{d^k}+\rho_k.
\end{align}
Summing up both sides of the latter inequality over $k=1,2,\ldots$ gives us
\begin{align*}
C_1\sum_{k=1}^m \norm{d^k}^2\le \varphi(x^1)-\varphi(x^{m+1})+C_2 \sum_{k=1}^m \sqrt{\rho_k}\norm{d^k}+\sum_{k=1}^m\rho_k\;\text{ for all }\;m\in\N.
\end{align*}
Since $\inf_{k\in\N}\varphi(x^k)>-\infty$ and $\sum_{k=1}^\infty \rho_k<\infty$, we get
\begin{align*}
C_1\sum_{k=1}^m \norm{d^k}^2\le C_3+C_2 \sum_{k=1}^m \sqrt{\rho_k}\norm{d^k}\;\text{ for all }\;m\in \N,
\end{align*}
where $C_3:=\varphi(x^1)-\inf_{k\in\N}\varphi(x^k) +\sum_{k=1}^\infty\rho_k<\infty$.
It implies by the Cauchy-Schwarz inequality that  
\begin{align*}
&C_1\sum_{k=1}^m\norm{d^k}^2\le C_3+C_2\sqrt{\sum_{k=1}^m \rho_k}\sqrt{\sum_{k=1}^m \norm{d^k}^2}\\
&\le C_3+C_2\sqrt{\sum_{k=1}^\infty \rho_k}\sqrt{\sum_{k=1}^m \norm{d^k}^2}=C_3+C_4\sqrt{\sum_{k=1}^m \norm{d^k}^2},
\end{align*}
where $C_4:=C_2\sqrt{\sum_{k=1}^\infty \rho_k}$. Therefore, we arrive at
\begin{align*}
C_1\brac{\sqrt{\sum_{k=1}^m \norm{d^k}^2}-\dfrac{C_4}{2C_1}}^2-\dfrac{C_4^2}{4C_1}=C_1\sum_{k=1}^m \norm{d^k}^2-C_4\sqrt{\sum_{k=1}^m \norm{d^k}^2}\le C_3\;\text{ for all }\;m\in\N,
\end{align*}
which gives us the desired condition $\sum_{k=1}^\infty \norm{d^k}^2<\infty$.\\
Suppose now that $\rho_k=\min\set{\dfrac{C_1^2}{4C_2^2},\dfrac{C_1}{4}}\varepsilon_k^2$ for any $k\in \N$ and show that \eqref{C1 C2} yields
\begin{align}\label{descent varphi0}
\dfrac{C_1}{4}\norm{d^k}^2\le \varphi(x^k)-\varphi(x^{k+1})\;\text{ whenever }\;k\in\N.
\end{align}
If $k\in\mathcal{N}$, inequality \eqref{descent varphi0} is obviously satisfied since both sides are zeros. Otherwise, we get 
\begin{align*}
\norm{d^k}=\norm{g^k}-\varepsilon_k> r_k+\varepsilon_k-\varepsilon_k=r_k\ge \varepsilon_k.
\end{align*}
Therefore, the choice of $\set{\rho_k}$ leads us to the inequalities
\begin{align*}
C_2\sqrt{\rho_k}\norm{d^k}+\rho_k\le C_2\dfrac{C_1}{2C_2}\varepsilon_k\norm{d^k}+\dfrac{C_1}{4}\varepsilon_k^2\le \dfrac{3C_1}{4}\norm{d^k}^2.
\end{align*}
Combining this with \eqref{C1 C2}, we arrive at
\begin{align*}
C_1\norm{d^k}^2\le  \varphi(x^k)-\varphi(x^{k+1})+\dfrac{3C_1}{4}\norm{d^k}^2,
\end{align*}
which implies in turn the fulfillment of \eqref{descent varphi0}. Summing up both sides of \eqref{descent varphi0} over $k=1,2,\ldots$ with taking into account that $\inf_{k\in\N}\varphi(x^k)>-\infty$ verifies $\sum_{k=1}^\infty\norm{d^k}^2<\infty$ and completes the proof.
\end{proof}

\begin{Remark}\rm\label{remark descent}
The choice of and arbitrary sequence $\set{\rho_k}$ satisfying $\sum_{k=1}^\infty\rho_k<\infty$ may not ensure the decreasing property of $\set{\varphi(x^k)}.$ However, if we select $\rho_k:=C\varepsilon_k^2$ for all $k\in\N$ with $C:=\min\set{\dfrac{C_1^2}{4C_2^2},\dfrac{C_1}{4}}$, the sequence $\set{\varphi(x^k)}$ {\em is decreasing} due to
\begin{align}\label{descent varphi}
\dfrac{C_1}{4}\norm{d^k}^2\le \varphi(x^k)-\varphi(x^{k+1})\;\text{ for all }\;k\in\N.
\end{align}
\end{Remark}\vspace*{0.05in}

Our next goal is to establish the {\em global convergence} of the sequence $\set{x^k}$ generated by Algorithm~\ref{pgm}. First we recall the notion of {\em forward-backward envelopes} introduced by Patrinos and Bemporad \cite{pb}  and then being largely used to the design and justification of numerical methods in various composite frameworks of optimization; see, e.g., \cite{BorisKhanhPhat,kmptjogo,kmptmp,stp,themelis18} with the references therein.

\begin{Definition}[\bf forward-backward envelopes] \rm Let $\varphi=f+g$ be given as in \eqref{eq:composite} and let $\lambda\in(0,\varrho^{-1})$. The {\em forward-backward envelope} (FBE) of $\varphi$ is given by
\begin{equation*}
\varphi_\lambda(x):= \inf_{y\in\R^n}\big\{f(x)+\dotproduct{\nabla f(x),y-x}+(2\lambda)^{-1}\norm{y-x}^2+g(y)\big\}.
\end{equation*}
\end{Definition}
The FBE $\varphi_\lambda$ is majorized by $\varphi$, i.e., $\varphi_\lambda(x)\le \varphi(x)$ for all $x\in\R^n$. This definition can be rewritten via the Moreau envelope \eqref{Moreau} of $g$  as
\begin{align*}
\varphi_\lambda(x)
&=f(x) - \frac{\lambda}{2}\|\nabla f(x)\|^2+ \inf_{y\in\R^n}\big\{g(y)+(2\lambda)^{-1}\norm{y-(x-\lambda \nabla f(x))}^2\big\}\\
&=f(x)-\dfrac{\lambda}{2}\norm{\nabla f(x)}^2+\inf_{y\in\R^n} \Phi_{\lambda,F_\lambda(x)}(y)\\
&= f(x) - \frac{\lambda}{2}\|\nabla f(x)\|^2+ e_\lambda g(x-\lambda\nabla f(x)).
\end{align*}
When $f$ is $\mathcal{C}^2$-smooth, it follows from the latter equality and Proposition~\ref{weakly and descent} that the FBE $\varphi_\lambda$ of $\varphi$ is $\mathcal{C}^1$-smooth with the gradient given by
\begin{align}\label{eq:grad FBE}
\nabla \varphi_\lambda(x) &= \lambda^{-1}(I-\lambda \nabla^2 f(x))(x- \text{\rm Prox}_{\lambda g}(x-\lambda \nabla f(x)))\nonumber\\
&=\lambda^{-1}(I-\lambda\nabla^2f(x))(x-T_\lambda(x)).
\end{align}
\noindent Consider the following approximation of the objective function $\varphi$:
\begin{align}\label{def surrogate}
\begin{cases}
\mathcal{F}_\lambda:\R^n\times \R\rightarrow\R,\\
\mathcal{F}_\lambda(x,\varepsilon)=\varphi_\lambda(x)+\mathscr{C}\varepsilon^2\text{ for all }x\in \R^n,\;\varepsilon\in\R,
\end{cases}
\end{align}
where $\mathscr{C}=\min\set{\dfrac{\lambda(1-\lambda\varrho)}{2},\dfrac{C_1^2}{4C_2^2},\dfrac{C_1}{4}}$ with $C_1,C_2$ defined in \eqref{def C1 C2}.\vspace*{0.05in}

The next proposition provides important properties of $\mathcal{F}_\lambda$ related to $\set{x^k}$, which are used below to establish convergence results for the iterative sequence $\set{x^k}$.

\begin{Proposition}\label{properties surrogate} Let $\varphi$ be defined in \eqref{eq:composite}, where $f$ is a $\mathcal{C}^2$-smooth function with $L$-Lipschitz gradient, and let $\set{x^k}$ be the sequence generated by Algorithm~\ref{pgm} with $\varepsilon_1\le r_1,\theta\le\mu$ and $\rho_k=C\varepsilon_k^2$ for all $k\in\N$, where $C$ is taken from 
Theorem~{\rm\ref{stationarity of pgm}}. Then for all $k\notin\N$, there are constants $a_1,a_2>0$ such that:

{\bf(i)} $\varphi(x^{k+1})\le \mathcal{F}_\lambda (x^k,\varepsilon_k)\le \varphi(x^k)+\mathscr{C}\varepsilon_k^2$.

{\bf(ii)} $\varphi(x^k)-\varphi(x^{k+1})\ge a_1 \norm{x^{k+1}-x^k}^2$.

{\bf(iii)} $\norm{\nabla \mathcal{F}_\lambda(x^k,\varepsilon_k)}\le a_2\norm{x^{k}-x^{k+1}}$.
\end{Proposition}

\begin{proof}
Pick $k\notin\mathcal{N}$ and deduce from the second inequality in (i) that $\varphi_\lambda(x^k)\le \varphi(x^k)$. It follows by the constructions of  $\mathscr{C},C$, and $\rho_k$ that
\begin{align}\label{A epsilon}
\min\Big\{\dfrac{\lambda(1-\lambda\varrho)}{2}\varepsilon_k^2,\rho_k \Big\}&=\min\Big\{\dfrac{\lambda(1-\lambda\varrho)}{2}\varepsilon_k^2,\dfrac{C_1}{4C_2^2}\varepsilon_k^2,\dfrac{C_1}{4}\varepsilon_k^2\Big\}=\mathscr{C}\varepsilon_k^2.
\end{align}
Since $k\notin\mathcal{N}$, the vector $x^{k+1}=p^k$ is a $\mathscr{C}\varepsilon_k^2$-solution to subproblem \eqref{subproblem}. Combining this with \eqref{subproblem}, \eqref{A epsilon}, and the $L$-descent property of $f$ as in \eqref{descent condition}, we get that
\begin{align*}
\varphi_\lambda(x^k)&=f(x^k)-\dfrac{\lambda}{2}\norm{\nabla f(x^k)}^2+\inf_{y\in\R^n} \Phi_{\lambda,F_\lambda(x^k)}(y)\\
&\ge f(x^k)-\dfrac{\lambda}{2}\norm{\nabla f(x^k)}^2+\Phi_{\lambda,F_\lambda(x^k)}(x^{k+1})-\mathscr{C}\varepsilon_k^2\\
&=f(x^k)-\dfrac{\lambda}{2}\norm{\nabla f(x^k)}^2+g(x^{k+1})+\dfrac{1}{2\lambda}\norm{x^{k+1}-(x^k-\lambda\nabla f(x^k))}^2-\mathscr{C}\varepsilon_k^2\\
&= f(x^k)+\dotproduct{\nabla f(x^k),x^{k+1}-x^k}+\dfrac{1}{2\lambda}\norm{x^{k+1}-x^k}^2+g(x^{k+1})-\mathscr{C}\varepsilon_k^2\\
&\ge f(x^{k+1})-\dfrac{L}{2}\norm{x^{k+1}-x^k}^2+\dfrac{1}{2\lambda}\norm{x^{k+1}-x^k}^2+g(x^{k+1})-\mathscr{C}\varepsilon_k^2\\
&\ge \varphi(x^{k+1})-\mathscr{C}\varepsilon_k^2,
\end{align*}
where the last inequality is deduced from $\lambda^{-1}>L$. Therefore, (i) is justified.

It follows from $\varepsilon_1\le r_1$ and $\theta\le \mu$ that $\varepsilon_k\le r_k$. By Proposition \ref{tk < 2 lambda}(iii), we have $t_k\le 2\lambda.$
Combining this with \eqref{descent varphi} and the relation $x^{k+1}=x^{k}+t_kd^k$, we get
\begin{align*}
\varphi(x^k)-\varphi(x^{k+1})&\ge \dfrac{C_1}{4}\norm{d^k}^2=\dfrac{C_1}{4t_k^2}\norm{x^{k+1}-x^k}^2\\
&\ge\dfrac{C_1}{16\lambda^2}\norm{x^{k+1}-x^k}^2,
\end{align*}
which justifies (ii) with $a_1:=C_1/16\lambda^2$.

To verify (iii), we first observe that $\nabla \mathcal{F}_\lambda(x^k,\varepsilon_k)=(\nabla\varphi_\lambda(x^k),2\mathscr{C}\varepsilon_k).$ Since $\nabla f$ is $L$-Lipschitz, it follows from Proposition~\ref{prop characterize l descent} that $\norm{\nabla^2 f(x^k)}\le L$, and hence
\begin{align}
\norm{I-\lambda\nabla^2 f(x^k)}\le \norm{I}+\lambda\norm{\nabla^2f(x^k)}\le 1+\lambda L.
\end{align}
Using the representation of $\nabla\varphi_\lambda$ in \eqref{eq:grad FBE} and the latter estimate give us 
\begin{align}\label{norm grad envelope}
\norm{\nabla \varphi_\lambda(x^k)}&=\lambda^{-1}\norm{(I-\lambda\nabla^2f(x^k))(x^k-T_\lambda(x^k))}\nonumber\\
&\le \lambda^{-1}\norm{I-\lambda \nabla^2 f(x^k)}\cdot\norm{x^k-T_\lambda(x^k)}\nonumber\\
&\le \lambda^{-1}(1+\lambda L)\norm{x^k-T_\lambda(x^k)}\nonumber\\
&\le (\lambda^{-1}+L)\norm{x^k-x^{k+1}}+\lambda^{-1}(1+\lambda L)\norm{x^{k+1}-T_\lambda(x^k)}.
\end{align}
Furthermore, it follows from $x^{k+1}=p^k$ and Proposition~\ref{tk < 2 lambda}(i) that $\lambda^{-1}\norm{x^{k+1}-T_\lambda(x^k)}\le \varepsilon_k$. Combining the latter with \eqref{norm grad envelope} ensures that 
\begin{align}\label{norm grad env 2}
\norm{\nabla \varphi_\lambda(x^k)}\le (\lambda^{-1}+L)\norm{x^k-x^{k+1}}+(1+\lambda L)\varepsilon_k.
\end{align}
Remembering that $\varepsilon_k\le r_k$ and $k\notin\mathcal{N}$ allows us to deduce from Remark~\ref{null remark 2}(iii) that $\varepsilon_k\le \norm{d^k}$. Since $t_k\ge \lambda$ by Proposition~\ref{tk < 2 lambda}, we get 
\begin{align}\label{epsilon xk xk+1}
\varepsilon_k\le \norm{d^k}=t_k^{-1}\norm{x^k-x^{k+1}}\le\lambda^{-1}\norm{x^{k}-x^{k+1}}.
\end{align}
Combining finally \eqref{norm grad env 2} and \eqref{epsilon xk xk+1} with the representation of $\nabla\mathcal{F}_\lambda$ gives us
\begin{align*}
\norm{\nabla \mathcal{F}_\lambda(x^k,\varepsilon_k)}&\le \norm{\nabla \varphi_\lambda(x^k)}+2\mathscr{C}\varepsilon_k\\
&\le (\lambda^{-1}+L)\norm{x^k-x^{k+1}}+(1+\lambda L)\varepsilon_k+2\mathscr{C}\varepsilon_k\\
 &\le (\lambda^{-1}+L)\norm{x^k-x^{k+1}}+(1+\lambda L+2\mathscr{C})\lambda^{-1}\norm{x^k-x^{k+1}}\\
 &=2(\lambda^{-1}+L+\mathscr{C}\lambda^{-1})\norm{x^{k}-x^{k+1}},
\end{align*}
which justifies (iii) with $a_2:=2(\lambda^{-1}+L+\mathscr{C}\lambda^{-1})$ and thus completes the proof.
\end{proof}

The next result inspired by \cite[Theorem~2]{attouch10} is useful to establish the convergence rates below.

\begin{Lemma}\label{lemma KL fran}
Let $\set{s_k}$ be a decreasing sequence of positive numbers, and let $\alpha>1$ and $\beta>0$ be some constants. Suppose that $s_k\downarrow 0$ as $k\to\infty$, and that for all  large $k\in\N$ we have
\begin{align}\label{condi s alpha}
s_k^\alpha\le\beta(s_{k-1}-s_{k}).
\end{align}
Then there exists a number $\gamma>0$ such that 
\begin{align*}
 s_k\le\gamma k^{-\frac{1}{\alpha-1}}\;\text{ whenever }\;k\in\N\;\text{  is sufficiently large}.
\end{align*}
\end{Lemma}
\begin{proof}
Let $N\in\N$ be such that \eqref{condi s alpha} holds for all $k\ge N$. Defining
\begin{align*}
\bar \mu:=\min\set{\dfrac{\alpha-1}{2\beta},(2^{\frac{\alpha-1}{\alpha}}-1)s^{1-\alpha}_{N-1}}>0,
\end{align*}
we are going to show that
\begin{align}\label{case 1}
s_k^{1-\alpha}-s_{k-1}^{1-\alpha}\ge \bar \mu\;\text{ for all }\;k\ge N.
\end{align}
Consider the two cases for $k.$\vspace*{0.03in}
 
\noindent\textit{Case~1}: $s_k^{-\alpha}>2s_{k-1}^{-\alpha}$. Since $1-\alpha<0$, we have $s_k^{1-\alpha}> 2^{\frac{\alpha-1}{\alpha}} s^{1-\alpha}_{k-1}$. Then the decreasing property of $\set{s_k}$ yields the desired result:
\begin{align*}
s_k^{1-\alpha}-s_{k-1}^{1-\alpha}> (2^{\frac{\alpha-1}{\alpha}}-1)s_{k-1}^{1-\alpha}\ge (2^{\frac{\alpha-1}{\alpha}}-1)s^{1-\alpha}_{N-1}\ge \bar \mu.
\end{align*}
\textit{Case~2}:  $s_k^{-\alpha}\le 2s_{k-1}^{-\alpha}$. It follows from \eqref{condi s alpha} and the decreasing property of $\set{s_k}$ that 
\begin{align*}
1 &\le \beta(s_{k-1}-s_k)s_k^{-\alpha}\le 2\beta (s_{k-1}-s_k)s_{k-1}^{-\alpha}\\
&\le 2 \beta \int_{s_k}^{s_{k-1}}s^{-\alpha}ds=\dfrac{2\beta}{1-\alpha}(s_{k-1}^{1-\alpha}-s_k^{1-\alpha}),
\end{align*} which implies by $\alpha>1$ that
\begin{align*}
s_k^{1-\alpha}-s_{k-1}^{1-\alpha}\ge  \dfrac{\alpha-1}{2\beta}\ge \bar \mu.
 \end{align*}
Fixing some $K>2N$ and summing up \eqref{case 1} over $k=N+1,\ldots,K$, we get that $s_K^{1-\alpha}-s_{N}^{1-\alpha}\ge \bar{\mu}(K-N)$. This allows us to verify the estimates
\begin{align*}
s_K\le \sbrac{s_{N}^{1-\alpha}+\bar{\mu}(K-N)}^{\frac{1}{1-\alpha}}\le \brac{\dfrac{\bar\mu }{2}}^{\frac{1}{1-\alpha}}K^{\frac{1}{1-\alpha}}
\end{align*}
and therefore complete the proof of the lemma. 
\end{proof}

Now we are ready to establish the global convergence and obtain convergence rates of our IPGM Algorithm~\ref{pgm} under the KL property of $\mathcal{F}_\lambda$ from 
\eqref{def surrogate}.

\begin{Theorem}\label{convergence rate ipgm}
Let $\varphi$ be defined in \eqref{eq:composite}, where $f$ is a $\mathcal{C}^2$-smooth function with $L$-Lipschitz gradient, and let $\set{x^k}$ be the sequence generated by Algorithm~\ref{pgm} satisfying $\varepsilon_1\le r_1,\theta\le\mu$  and $\rho_k=C\varepsilon_k^2$ for all $k\in\N$ with $C$ taken from Theorem~{\rm\ref{stationarity of pgm}}. Assume further that the set $\N\setminus \mathcal{N}$ is infinite, that $\set{x^k}$ has an accumulation point $\bar x$, and that the function $\mathcal{F}_\lambda$ in \eqref{def surrogate} satisfies the KL property from Definition~{\rm\ref{KL ine}} at $(\bar x,0)$. Then we have the assertions:

{\bf(i)} $\bar x$ is a stationary point of $\varphi$ and the sequence $\set{x^k}$ converges to $\bar x$  having finite length, i.e., 
\begin{align*}
\sum_{k=1}^\infty \norm{x^{k+1}-x^k}<\infty.
\end{align*}

{\bf(ii)} If in addition the KL property of $\mathcal{F}_\lambda$ at $(\bar x,0)$ holds with $\psi(t):=Mt^{1-q}$ for some $M>0$ and $q\in (0,1)$, then the following convergence rates are guaranteed for the sequence $\set{z^k}$ from \eqref{0 zk xjk}:

$\bullet$ For $q\in(0,1/2]$, the sequence $\set{z^k}$ converges linearly to $\bar x$.

$\bullet$ For $q\in(1/2,1)$, there exists a positive constant $\gamma$ such that 
\begin{align*}
\norm{z^k-\bar x}\le \gamma k^{-\frac{1-q}{2q-1}}\;\text{ for sufficiently large }\;k\in\N.
\end{align*}
\end{Theorem}
\begin{proof}
(i) It follows from Proposition~\ref{tk < 2 lambda} that $\set{x^k}$ is a sequence generated by Algorithm~\ref{scheme continuous mapping} for finding zeros of the continuous mapping $G_\lambda$. Then the imposed assumptions guarantee that all the conditions of Proposition~\ref{0 prop non null sequence} are satisfied. Suppose that $\N\setminus\mathcal{N}=\set{j_1,j_2,\ldots}$, and let $\set{z^k}$ be taken from\eqref{0 zk xjk} while $\set{\sigma_k}$ is defined by $\sigma_k:=\varepsilon_{j_k}$ for all $k\in\N$. Since $\bar x$ is an accumulation point of $\set{x^k}$, it is also an accumulation point of $\set{z^k}$ by 
Proposition~\ref{0 prop non null sequence}. Therefore, we can find an infinite index subset $I\subset\N$ such that $z^k\overset{I}{\rightarrow}\bar x$. Fix any $k\in\N$. By taking into account that $j_k\notin\mathcal{N}$ and using the construction of $\set{z^k},\set{\varepsilon_{j_k}}$ and $z^{k+1}=x^{j_k+1}$ in Proposition~\ref{0 prop non null sequence}, we deduce from Proposition~\ref{properties surrogate} that 
\begin{align}
\varphi(z^{k+1})\le \mathcal{F}_\lambda (z^k,\sigma_k)\le \varphi(z^k)+\mathscr{C}\sigma_k^2,\label{3 ine}\\
\varphi(z^k)-\varphi(z^{k+1})\ge a_1 \norm{z^{k+1}-z^k}^2,\label{>a}\\
\norm{\nabla \mathcal{F}_\lambda(z^k,\sigma_k)}\le a_2\norm{z^{k+1}-z^{k}}\label{<b}
\end{align}
whenever $k\in\N$, where $a_1,a_2$ are some fixed positive constants. It follows from \eqref{>a} and $z^k\ne z^{k+1}$ for all $k\in\N$ in Proposition~\ref{0 prop non null sequence} that $\set{\varphi(z^k)}$ is strictly decreasing. Combining this with the lower semicontinuity of $\varphi$ gives us the relationships
\begin{align}\label{inf >}
\inf_{k\in\N}\varphi(x^k)=\inf_{k\in\N}\varphi(z^k)=\lim_{k\rightarrow\infty}\varphi(z^k)=\lim_{k\overset{I}{\rightarrow}\infty}\varphi(z^k)\ge \varphi(\bar x)>-\infty.
\end{align} 
This shows together with the imposed assumptions tells us that all the requirements of Theorem~\ref{stationarity of pgm} are satisfied. Thus it follows from assertions (i) and (ii) of Theorem~\ref{stationarity of pgm} that $\bar x$ is a stationary point of $\varphi$, that $\varepsilon_k\downarrow0$, and that $\norm{x^k-x^{k+1}}\rightarrow 0$ as $k\to\infty$. The constructions of $\set{z^k}$ and $\set{\sigma_k}$ together with $z^{k+1}=x^{j_k+1}$ for all $k\in\N$ yield $\norm{z^{k+1}-z^{k}}\rightarrow0$ and $\sigma_k\downarrow 0$. Combining the latter with \eqref{3 ine}, the continuity of $\mathcal{F}_\lambda$, and the condition $\varphi_\lambda(x)\le \varphi(x)$ for all $x\in\R^n$, we get
\begin{align*}
\lim_{k\rightarrow\infty}\varphi(z^k)=\lim_{k\overset{I}{\rightarrow}\infty}\varphi(z^{k+1})\le\lim_{k\overset{I}{\rightarrow}\infty} \mathcal{F}_\lambda(z^k,\sigma_k)=\mathcal{F}_\lambda(\bar x,0)=\varphi_\lambda(\bar x)\le \varphi(\bar x),
\end{align*}
which brings us by taking \eqref{inf >} into account to
\begin{align}\label{limit F_lambda}
\lim_{k\rightarrow\infty}\varphi(z^k)=\lim_{k\rightarrow\infty}\mathcal{F}_\lambda(z^k,\sigma_k)=\varphi(\bar x)=\mathcal{F}_\lambda(\bar x,0).
\end{align}
Involving now the KL property of $\mathcal{F}_\lambda$ at $\bar y:=(\bar x,0)$, we find numbers $\eta>0,\nu>0$ and a continuous concave function $\psi:[0,\eta)\rightarrow[0,\infty)$ satisfying the conditions: $\psi(0)=0$, $\psi$ is $\mathcal{C}^1$ on $(0,\eta)$, $\psi'(s)>0$ for all $s\in (0,\eta)$, and
for all $y\in \R^{n}\times \R$ such that $\norm{y-\bar y}<\nu$ with $ \varphi(\bar x)<\mathcal{F}_\lambda(y)< \varphi(\bar x)+\eta$ we have
\begin{align}\label{kll 2}
\psi'(\mathcal{F}_\lambda(y)-\varphi(\bar x))\norm{\nabla\mathcal{F}_\lambda (y)}\ge 1.
\end{align}
For any $k\in\N$, set $y^k:=(z^k,\sigma_k)$ and by using \eqref{limit F_lambda}, \eqref{3 ine}, and $\sigma_k\downarrow0$ find $k_1\in\N$ such that
\begin{align}\label{condition for using KL 2}
\varphi(\bar x)< \varphi(z^k)\le \mathcal{F}_\lambda(y^{k-1})<\varphi(\bar x)+\eta\quad \text{ and }\quad \sigma_k<\nu/4\;\text{ whenever }\;k\ge k_1.
\end{align}
Since $\bar x$ is an accumulation point of $\set{z^k}$ with $\norm{z^{k+1}-z^{k}}\rightarrow0$, and since $\psi$ is continuous at $0$ with $\psi(0)=0$, suppose without loss of generality that
\begin{align}\label{further condition}
\norm{z^{k_1-1}-\bar x}<\nu/4,\quad \norm{z^{k_1}-z^{k_1-1}}<\nu/8,\;\mbox{ and }\;a_2a_1^{-1}\psi(\varphi(z^{k_1})-\varphi(\bar x))<\nu/4.
\end{align}
We claim that if $k\ge k_1$ and $\norm{y^{k-1}-\bar y}<\nu$, then
\begin{align}\label{ine claim 1}
2\norm{z^{k+1}-z^{k}}\le \phi_k+\norm{z^{k}-z^{k-1}},
\end{align} 
where $\phi_k:=a_2a_1^{-1}\sbrac{\psi(\varphi(z^k)-\varphi(\bar x))-\psi(\varphi(z^{k+1})-\varphi(\bar x))}$. To verify this, fix such $k$. Since $\norm{y^{k-1}-\bar y}<\nu$ and $k\ge k_1$, it follows from the first assertion in \eqref{condition for using KL 2} that inequality \eqref{kll 2} holds for $y=y^{k-1},$ i.e., 
\begin{align}\label{psi ' F lambda}
\psi'(\mathcal{F}_\lambda(y^{k-1})-\varphi(\bar x))\norm{\nabla \mathcal{F}_\lambda(y^{k-1})}\ge 1.
\end{align}
The latter implies by \eqref{<b} that
\begin{align}\label{psi' 2}
\psi'(\mathcal{F}_\lambda(y^{k-1})-\varphi(\bar x))\ge \dfrac{1}{\norm{\nabla \mathcal{F}_\lambda(y^{k-1})}}\ge \dfrac{1}{a_2\norm{z^{k}-z^{k-1}}}.
\end{align}
Since $\psi$ is concave, $\psi'$ is nonincreasing. It follows from \eqref{3 ine} that
$\mathcal{F}_\lambda(y^{k-1})-\varphi(\bar x)\ge \varphi(z^k)-\varphi(\bar x)$, and the nonincreasing property implies therefore that
\begin{align*}
\dfrac{1}{a_2\norm{z^{k}-z^{k-1}}}\le \psi'(\mathcal{F}_\lambda(y^{k-1})-\varphi(\bar x))\le \psi'(\varphi(z^k)-\varphi(\bar x)).
\end{align*}
Combining the latter with the concavity of $\psi$ and condition \eqref{>a} brings us to the inequalities 
\begin{align*}
&\psi(\varphi(z^k)-\varphi(\bar x))-\psi(\varphi(z^{k+1})-\varphi(\bar x))\ge\psi'(\varphi(z^k)-\varphi(\bar x))(\varphi(z^k)-\varphi(z^{k+1}))\\
&\ge \psi'(\varphi(z^k)-\varphi(\bar x))a_1\norm{z^{k+1}-z^{k}}^2\ge \dfrac{a_1\norm{z^{k+1}-z^{k}}^2}{a_2\norm{z^k-z^{k-1}}}.
\end{align*}
This can be equivalently written as
\begin{align*}
\norm{z^{k+1}-z^k}^2\le\phi_k\norm{z^k-z^{k-1}},
\end{align*}
which implies in turn by using $2\sqrt{uv}\le u+v$ that
\begin{align*}
2\norm{z^{k+1}-z^k}\le \phi_k+\norm{z^k-z^{k-1}}
\end{align*}
and thus verifies the above claim.

Now we show that inequality \eqref{ine claim 1} holds for all $k\ge k_1$ by induction. This is obviously true for $k=k_1$ since $\norm{y^{k_1-1}-\bar y}\le \norm{z^{k_1-1}-\bar x}+\sigma_{k_1}< \nu/2$. Suppose that it holds for $k_1,k_1+1,\ldots,K-1.$ Summing up inequality \eqref{ine claim 1} over $k=k_1,k_1+1,\ldots,K-1$, we get
\begin{align*}
2\sum_{k=k_1}^{K-1}\norm{z^{k+1}-z^{k}}&\le \sum_{k=k_1}^{K-1}\norm{z^{k}-z^{k-1}}+\sum_{k=k_1}^{K-1}\phi_k\\
&\le \sum_{k=k_1}^{K-1}\norm{z^{k+1}-z^{k}}+\norm{z^{k_1}-z^{k_1-1}}+\dfrac{a_2}{a_1}\psi(\varphi(z^{k_1})-\varphi(\bar x)),
\end{align*}
which implies in turn that
\begin{align}\label{letting k infi}
\sum_{k=k_1}^{K-1}\norm{z^{k+1}-z^{k}}\le \norm{z^{k_1}-z^{k_1-1}}+\dfrac{a_2}{a_1}\psi(\varphi(z^{k_1})-\varphi(\bar x)).
\end{align}
Employing the triangle inequality gives us 
\begin{align*}
\norm{z^{K-1}-\bar x}&\le \norm{z^{k_1}-\bar x}+\sum_{k=k_1}^{K-1}\norm{z^{k+1}-z^{k}}\\
&\le  \norm{z^{k_1-1}-\bar x}+2\norm{z^{k_1}-z^{k_1-1}}+\dfrac{a_2}{a_1}\psi(\varphi(z^{k_1})-\varphi(\bar x))< \dfrac{3\nu}{4}.
\end{align*}
Since $\sigma_{K-1}< \nu/4$ by \eqref{condition for using KL 2}, we get that 
$$\norm{y^{K-1}-\bar y}\le \norm{z^{K-1}-\bar x}+\sigma_{K-1}<\nu,$$ 
and thus inequality \eqref{ine claim 1} holds for $k=K$. Letting $K\rightarrow\infty$ in \eqref{ine claim 1} gives us 
\begin{align*}
\sum_{k=1}^\infty\norm{z^{k+1}-z^{k}}<\infty.
\end{align*}
The latter means that $\set{z^k}$ is a Cauchy sequence, and hence it converges to $\bar x$. Therefore, $\set{x^k}$ converges to $\bar x$ as well, and moreover we have
\begin{align*}
&\sum_{k=1}^\infty\norm{x^{k+1}-x^{k}}=\sum_{k\in\mathcal{N}}\norm{x^{k+1}-x^{k}}+\sum_{k\notin \mathcal{N}}\norm{x^{k+1}-x^{k}}\\
&=\sum_{k\in\N}\norm{x^{j_k+1}-x^{j_k}}=\sum_{k=1}^\infty\norm{z^{k+1}-z^{k}}<\infty,
\end{align*}
which verifies assertion (i). 

(ii) For any $p\in\N,$ set $s_p:=\sum_{k=p}^\infty\norm{z^{k+1}-z^{k}}$, which is a finite number. The triangle inequality yields $\norm{z^p-\bar x}\le s_p$, it suffices to establish the convergence rates for $s_p$. To proceed, take any $p\ge k_1$, where $k_1$  satisfies \eqref{condition for using KL 2} and \eqref{further condition}. It follows from the proof of (i) that
\begin{align*}
2\norm{z^{k+1}-z^k}\le \phi_k+\norm{z^k-z^{k-1}}\text{ for all }k\ge p.
\end{align*}
Summing up the above inequality over $k=p,p+1,\ldots$, we get
\begin{align*}
2\sum_{k=p}^\infty\norm{z^{k+1}-z^k}&\le \sum_{k=p}^\infty\norm{z^k-z^{k-1}}+\sum_{k=p}^\infty\phi_k\\
&\le\sum_{k=p}^\infty\norm{z^{k+1}-z^{k}}+\norm{z^{p}-z^{p-1}}+\dfrac{a_2}{a_1}\psi(\varphi(z^p)-\varphi(\bar x)),
\end{align*}
which implies in turn that 
\begin{align}\label{condition sp norm}
s_p\le\norm{z^{p}-z^{p-1}}+\dfrac{a_2}{a_1}(\varphi(z^p)-\varphi(\bar x))^{1-q}.
\end{align}
It follows from \eqref{psi ' F lambda} by using $\psi'(t)=(1-q)Mt^{-q}$ for all $t>0$ that $$(\mathcal{F}_\lambda(y^{p-1})-\varphi(\bar x))^q\le M(1-q)\norm{\nabla \mathcal{F}_\lambda(y^{k-1})}.$$ Combining this with   $\varphi(z^p)\le \mathcal{F}_\lambda(y^{p-1})$ from \eqref{3 ine} and $q\in (0,1),$ we get
\begin{align*}
\dfrac{a_2}{a_1}(\varphi(z^{p})-\varphi(\bar x))^{1-q}&\le \dfrac{a_2}{a_1}(\mathcal{F}_\lambda(y^{p-1})-\varphi(\bar x))^{1-q}\\
&\le \dfrac{a_2}{a_1}(M(1-q)\norm{\nabla \mathcal{F}_\lambda(y^{p-1})})^{\frac{1-q}{q}}.
\end{align*}
The latter together with \eqref{condition sp norm} and \eqref{<b} gives us 
\begin{align}
s_p&\le \norm{z^{p}-z^{p-1}}+\dfrac{a_2}{a_1}(M(1-q)\norm{\nabla \mathcal{F}_\lambda(y^{p-1})}))^{\frac{1-q}{q}}\nonumber\\
&\le \norm{z^{p}-z^{p-1}}+\dfrac{a_2}{a_1}(M(1-q)a_2\norm{z^{p}-z^{p-1}})^{\frac{1-q}{q}}\text{ for all }p\ge k_1.\label{sp <=}
\end{align}
We consider the two possible cases for $q.$

\medskip

\noindent\textit{Case~1}: $q\in(0,1/2]$\textit{, or equivalently,} $(1-q)/q\ge 1$. Since $\norm{z^{k+1}-z^k}\rightarrow0$ in this case, we can assume that $M(1-q)a_2\norm{z^p-z^{p-1}}<1$ for all $k\ge p_1$ and thus deduce from \eqref{sp <=} that
\begin{align*}
s_p&\le \norm{z^{p}-z^{p-1}}+\dfrac{a_2}{a_1}(M(1-q)a_2\norm{z^{p}-z^{p-1}})\\
&\le \brac{1+\dfrac{M(1-q)a_2^2}{a_1}}\norm{z^{p}-z^{p-1}}=\alpha (s_{p-1}-s_p),
\end{align*}
where $\alpha:=\brac{1+\dfrac{M(1-q)a_2^2}{a_1}}$. The above estimate can be rewritten as 
\begin{align*}
s_p\le \dfrac{\alpha}{1+\alpha}s_{p-1}\text{ for all }p\ge k_1,
\end{align*}
which implies that $\set{s_k}$ linearly converges to $0.$ Therefore, $\set{z^k}$ linearly converges to $\bar x$ as $k\to\infty$.

\medskip

\noindent \textit{Case~2}: $q\in(1/2,1),$ \textit{or equivalently,} $(1-q)/q\in (0,1)$. It follows from \eqref{sp <=} in this case that 
\begin{align*}
s_p\le \brac{1+\dfrac{a_2}{a_1}(M(1-q)a_2)^{\frac{1-q}{q}}}\norm{z^{p}-z^{p-1}}^{\frac{1-q}{q}},
\end{align*}
which can be rewritten as
\begin{align*}
s_p^{\frac{q}{1-q}}\le \beta(s_{p-1}-s_p)\text{ for all }p\ge k_1,\;\mbox{ with }\;\beta:=\sbrac{1+\dfrac{a_2}{a_1}(M(1-q) a_2)^{\frac{1-q}{q}}}^{\frac{q}{1-q}}.
\end{align*}
By Lemma~\ref{lemma KL fran} with $\alpha:=q/(1-q)$, there exists a positive constant $\gamma$ such that 
\begin{align*}
  s_k\le \gamma k^{-\frac{1-q}{2q-1}}\;\text{ for sufficiently large }\;k\in\N.
\end{align*}
This completes the proof of the theorem.
\end{proof}

\begin{Remark}[comments on the KL assumption]\rm\label{remark KL}
Let us now discuss some efficient conditions on $f$ and $g$, which guarantee the KL property of $\mathcal{F}_\lambda$ given in \eqref{def surrogate} by
 \begin{align*}
\mathcal{F}_\lambda(x,\varepsilon)= f(x) - \frac{\lambda}{2}\|\nabla f(x)\|^2+ e_\lambda g(x-\lambda\nabla f(x))+\mathscr{C}\varepsilon^2.
\end{align*}

{\bf(i)} \textit{$f$ is analytic and $g$ is semialgebraic.} Indeed, it follows from Remark~\ref{semi algebraic} that $x\mapsto e_\lambda g(x)$ is semialgebraic, and thus is subanalytic. Since $f$ is analytic, \cite[Proposition~2.2.3]{krantz02} tells us that all the partial derivatives of $f$ is analytic. Using \cite[Proposition~2.2.8]{krantz02}, we deduce that $x\mapsto\norm{\nabla f(x)}^2$ is analytic, and thus $(x,\varepsilon)\mapsto f(x)-\dfrac{\lambda}{2}\norm{\nabla f(x)}^2+\mathscr{C}\varepsilon^2$ is analytic. Since all the partial derivatives of $f$ is analytic, Remark~\ref{remark subana} also tells us that $\nabla f$ is 
subanalytic. Since the functions $x\mapsto e_\lambda g(x)$ and $x\mapsto x-\lambda \nabla f(x)$ are continuous, we deduce from  Remark~\ref{remark subana}(iii) that $e_\lambda g(x-\lambda\nabla f(x))$ is subanalytic, continuous, and so is $\mathcal{F}_\lambda$. By the continuity of $\mathcal{F}_\lambda$, the KL property of $\mathcal{F}_\lambda$ is satisfied for $\psi(t):=Mt^{1-q}$ with $M>0$ and $q\in[0,1)$.

{\bf(ii)} \textit{$f$ and $g$ are semialgebraic functions.} Indeed, similarly to (i), $e_\lambda g$ is a semialgebraic function. By Remark~\ref{semi algebraic}(v), the mapping $x\mapsto\nabla f(x)$ is also semialgebraic. Since $\norm{\cdot}^2$ is a polynomial function, Remark~\ref{semi algebraic}(i,ii) tells us again that $\mathcal{F}_\lambda$ is semialgebraic. Thus it is subanalytic by Remark\ref{semi algebraic}(iv). Therefore, the KL property of $\mathcal{F}_\lambda$ is satisfied for $\psi(t):=Mt^{1-q}$ with $M>0$ and $q\in[0,1)$.
\end{Remark}

\begin{Remark}[comparison between IPGM and iFB] \rm
Let us now compare the major issues in the algorithm description and the obtained results for IPGM given in Theorems~\ref{stationarity of pgm} and \ref{convergence rate ipgm} with that of the {\em inexact forward-backward} (iFB) algorithm in \cite[Theorem~8]{bonettini20} to solve problem \eqref{eq:composite} when the function $g$ is {\em convex}. The iFB algorithm from \cite{bonettini20} can be written as follows:

\begin{longfbox}
\begin{Algorithm}[inexact forward-backward (iFB)]\label{ifb}
\setcounter{Step}{-1} 
\begin{Step}\rm
Choose $\lambda\in(0,L^{-1})$, a starting point $x^1$ and a sequence of errors $\omega_k$ such that $\sum_{k=1}^\infty\sqrt{\omega_k}<\infty$. Set $k:=1$.
\end{Step}
\begin{Step}\rm
Compute $p^k$ such that 
\begin{align}
&\Phi_{\lambda,F_\lambda(x^k)}(p^k)\le \inf  \Phi_{\lambda,F_\lambda(x^k)}+ \omega_k \label{stopphi},\\
&\dotproduct{\nabla f(x^k),p^k-x^k} +(2\lambda)^{-1}\norm{p^k-x^k}^2+g(p^k)< g(x^k).\label{condi less}
\end{align}
\end{Step}
\begin{Step}\rm
Set $x^{k+1}:=p^k.$ Increase $k$ by $1$ and go back to Step~1.
\end{Step}
\end{Algorithm}
\end{longfbox}\vspace*{0.1in}

We comments on the following issues.\vspace*{0.03in}

{\bf(i)} \textit{Algorithm description}. The main difference between IPGM and iFB is the choice of {\em stopping criterion} for the subproblem of finding approximate solution $p^k$ to $T_\lambda(x^k)$ at each iteration. Although both algorithms use the condition \eqref{stopphi} for finding $p^k$, the selections of $\omega_k$ are different.  The error $\omega_k$ in iFB needs to satisfy the condition $\sum_{k=1}^\infty\sqrt{\omega_k}<\infty$. On the other hand, it follows from Algorithm \ref{pgm} and Theorem~\ref{stationarity of pgm} that IPGM considers two selections for $\omega_k$ which are either
\begin{align}\label{1st choice}
\omega_k:&=\min\set{\dfrac{\lambda\varepsilon_k^2}{2},\rho_k}\;\text{ with }\;\sum_{k=1}^\infty \rho_k<\infty,\;\text{ or }\\
\label{2nd choie}
\omega_k:&=\mathscr{C}\varepsilon_k^2\;\text{ with }\;\mathscr{C}:=\min\set{\dfrac{\lambda}{2},\dfrac{C_1^2}{4C_2^2},\dfrac{C_1}{4}}, 
\end{align}
where $C_1,C_2$  are defined in Proposition~\ref{prop dk goes 0}. Step~\ref{Step 2} of Algorithm~\ref{pgm} shows that $\varepsilon_k$ {\em doesn't need to decrease} after an iteration. This number is reduced only when $\norm{g^k}$ is sufficiently small, which holds when $x^k$ is near the exact solution of problem \eqref{eq:composite}. When $k\rightarrow\infty$, the convergence rate to $0$ of $\set{\rho_k}$ in the first choice \eqref{1st choice} is slower than that of iFB. In the second selection \eqref{2nd choie}, the error $\omega_k$ of the subproblem is independent with $k$. The approximate solution $p^k$ for the subproblem of iFB needs to satisfy condition \eqref{condi less}, which ensures the sufficient decrease property of the sequence $\set{\varphi(x^k)}$. This property is guaranteed for IPGM if the second choice of error \eqref{2nd choie} is employed; see
Remark~\ref{remark descent}. 

{\bf(ii)} \textit{Stationary accumulation points}. Theorem~8(i) of \cite{bonettini20} requires that the smooth part $f$ has a Lipschitz gradient, that the function $\varphi$ is bounded from below, and that the sequence $\set{x^k}$ is bounded. On the other hand, our Theorem~\ref{stationarity of pgm} requires that the smooth part satisfies only the $L$-descent condition, which is strictly weaker than the Lipschitz continuity of $\nabla f$. This allows us to work with broader classes of functions by replacing the usual Euclidean distance by the {\em Bregman distance}. The boundedness of $\set{x^k}$ and the boundedness from below of $\varphi$ are not required in Theorem~\ref{stationarity of pgm}. The latter property is relaxed to $\inf_{k\rightarrow\infty}\varphi(x^k)>-\infty$.

{\bf(iii)} \textit{Global convergence.} Theorem~8(ii) of \cite{bonettini20} requires that the smooth part $f$ is analytic and  that the nonsmooth part $g$ is subanalytic. On the other hand, the global convergence analysis in our Theorem~\ref{convergence rate ipgm} is based on the KL property of $\mathcal{F}_\lambda$. It follows from Remark~\ref{remark KL} that the latter property holds when $g$ is {\em semialgebraic} and $f$ is {\em either analytic}, or {\em semialgebraic}.

{\bf(iv)} \textit{Convergence rate.} The rate of convergence for $\set{x^k}$ is established in Theorem \ref{convergence rate ipgm} under the KL property of $\mathcal{F}_\lambda$, while not any type of convergence rate for iFB is given in \cite{bonettini20}. Instead, \cite[Theorem~9]{bonettini20} presents a convergence rate for function values in the variable metric inexact linesearch-based algorithm.
\end{Remark}\vspace*{-0.2in}

\section{Numerical Experiments}\label{sec:6}

In this section, we compare the numerical efficiency of our IPGM  with iFB in Algorithm~\ref{ifb} for a variant of the {\em nonconvex image restoration problem} taken from \cite{stp}. To proceed, given a vector $b \in \R^n$, matrices $A\in \R^{n\times n}$, and $B \in \R^{m\times n}$ together with a number $\gamma>0$, we consider the following problem:
\begin{equation}\label{ncvimage}
\text{minimize } \quad \varphi(x):= \sum_{i=1}^n \log(1+(Ax-b)_i^2) +\gamma \|Bx\|_1 \quad 
\text{subject to} \quad x \in \R^n,
\end{equation}
where $\norm{u}_1:=\sum_{i=1}^n|u_i|$ for any $u\in\R^n.$

It can be seen thatn\eqref{ncvimage} is a special case of \eqref{eq:composite} with $f(x):=\sum_{i=1}^n \log(1+(Ax-b)_i^2) $ and $g(x):=\gamma \|Bx\|_1$. Indeed, $g$ is clearly convex while $f$ is continuously differentiable with the gradient 
\begin{align*}
\nabla f(x)=2A^T u,\;\text{ where }\;u_i=\frac{(Ax-b)_i}{1+(Ax-b)_i^2},\;i=1,\ldots,n, \quad \text{for all }\; x \in \R^n.
\end{align*}
The smooth part $f$ is also considered in \cite[Section~4.1]{bonettini20} for an image deblurring and denoising problem. It follows from \cite[(4.2)]{bonettini20} that $\nabla f$ is Lipschitz continuous with constant $L=2\norm{A}_1\norm{A}_\infty$, where 
\begin{align*}
\norm{A}_1=\max_{1\le j\le n}\sum_{i=1}^n|a_{ij}|,\quad \norm{A}_\infty=\max_{1\le i\le n}\sum_{j=1}^n|a_{ij}|,\quad A=[a_{ij}]_{i,j=\overline{1,n}}.
\end{align*}
Moreover, $f$ is analytic it is expressed as finite sums and compositions of analytic functions, $g$ is semialgebraic and $\varphi$ is bounded from below by $0$ Therefore, the assumptions on the objective function in Theorem~\ref{convergence rate ipgm} and \cite[Theorem~8]{bonettini20} are satisfied, and so both IPGM and iFB are applicable to solve \eqref{ncvimage}. 

Problem \eqref{ncvimage} is studied in \cite{stp} when the matrix $B$ is square and orthogonal. In this section, we consider $A\in\R^{n\times n},\;b\in\R^n,\;B\in\R^{m\times n}$ generated randomly, with i.i.d. (identically and independent distributed) standard Gaussian entries, where $m,n$ need not to be equal. The number $\gamma$ is chosen as either $10^{-3}$, or $10^{-6}$. We provide two numerical experiments to illustrate how the difference in the stopping criteria for the subproblems of iFB and IPGM affect their performance. 

\begin{itemize}
\item In the first experiment, we let iFB run for 2000 iterations and record the function value obtained by this method. Then IPGM runs until its function value is lower than the recorded one of iFB.

\item In the second experiment, we let both algorithms run until $\norm{g^k}\le 10^{-1}$, where $g^k:=\dfrac{1}{\lambda}\norm{x^k-p^k}$. We also stop the algorithms if they either reach the maximum number of iterations of $2,000,000$, or the time limit of $4000$ seconds.
\end{itemize}
More detailed information for the testing data used in the experiments is presented in Table~\ref{testing data}, where "TN" stands for test number.

\begin{center}
\begin{table}[H]
\centering
\begin{tabular}{|llll|llll|} 
\hline
\multicolumn{4}{|c|}{Experiment 1}  & \multicolumn{4}{c|}{Experiment 2}\\ 
\hline
\multicolumn{1}{|c}{TN} & \multicolumn{1}{c}{m} & \multicolumn{1}{c}{n} & \multicolumn{1}{c|}{$\gamma$} & \multicolumn{1}{c}{TN} & \multicolumn{1}{c}{m} & \multicolumn{1}{c}{n} & \multicolumn{1}{c|}{$\gamma$}  \\ 
\hline
1 & 200& 200& $10^{-3}$    & 1& 200& 200& $10^{-3}$     \\
2 & 400& 400& $10^{-3}$    & 2& 400& 400& $10^{-3}$     \\
3 & 800& 800& $10^{-3}$    & 3& 800& 800& $10^{-3}$     \\
4 & 1600  & 1600  & $10^{-3}$    & 4& 1600  & 1600  & $10^{-3}$     \\
5 & 200& 800& $10^{-3}$    & 5& 200& 800& $10^{-3}$     \\
6 & 400& 1600  & $10^{-3}$    & 6& 400& 1600  & $10^{-3}$     \\
7 & 800& 200& $10^{-3}$    & 7& 800& 200& $10^{-3}$     \\
8 & 1600  & 400& $10^{-3}$    & 8& 1600  & 400& $10^{-3}$     \\
9 & 200& 200& $10^{-6}$    & 9& 200& 200& $10^{-6}$     \\
10& 400& 400& $10^{-6}$    & 10& 400& 400& $10^{-6}$     \\
11& 800& 800& $10^{-6}$    & 11& 800& 800& $10^{-6}$     \\
12& 1600  & 1600  & $10^{-6}$    & 12& 1600  & 1600  & $10^{-6}$     \\
13& 200& 800& $10^{-6}$    & 13& 200& 800& $10^{-6}$     \\
14& 400& 1600  & $10^{-6}$    & 14& 400& 1600  & $10^{-6}$     \\
15& 800& 200& $10^{-6}$    & 15& 800& 200& $10^{-6}$     \\
16& 1600  & 400& $10^{-6}$    & 16& 1600  & 400& $10^{-6}$     \\
\hline
\end{tabular}
\caption{Testing data}
\label{testing data}
\end{table}
\end{center} \vspace*{-0.2in}
Now we describe the setting for both algorithms in the tests. Firstly, the proximal parameter is chosen as $\lambda=1/2L$ to ensure that $\lambda\in(0,L^{-1})$ for both algorithms. The starting point is chosen as $x^1:=0_{\R^n}$. At any iteration of each algorithm, the step of finding an approximation proximal point $p^k$ satisfying \eqref{stopphi} is equivalent to solving approximately the optimization problem 
\begin{equation}\label{primalncv}
\min \quad \Phi_{\lambda, F_\lambda(x^k)} (p):=(2\lambda)^{-1}\norm{p-F_\lambda (x^k)}^2 +\gamma\norm{Bx}_1.
\end{equation}The dual problem of \eqref{primalncv} is given as
\begin{align}\label{dual problem num}
\max \quad \Psi_{\lambda,k}(y):&=-\dfrac{\lambda}{2}\norm{B^*y}^2+\dotproduct{BF_\lambda(x^k),y}-\delta_{\gamma\mathbb{B}^\infty}(y),
\end{align}
where $\mathbb{B}^\infty:=\set{y\in\R^m\;|\;\max_{i=1,\ldots,m}|y_i|\le 1}$. In our tests, the dual sequence $\{y^l\}$ to solve problem \eqref{dual problem num} is generated by means of the  Fast Iterative Shrinkage-Thresholding Algorithm (FISTA) from \cite{beck09} and the primal sequence in it is chosen as $p^l:= -\lambda B^*y^l+F_\lambda(x^k)$. The following main stopping criterion is used in both IPGM and iFB:
\begin{align}\label{stop by dual gap}
\Phi_{\lambda,F_\lambda(x^k)}(p^l) - \Psi_{\lambda,k}(y^l) \leq \omega_k,
\end{align}
where the choice of $\omega_k$ depends on the method. Namely, the settings for IPGM and iFB are as follows:

\begin{itemize}
\item IPGM: $\omega_k= \mathscr{C}\varepsilon_k^2$, where $\mathscr{C}:=\min\set{\dfrac{\lambda(1-\lambda\varrho)}{2},\dfrac{C_1^2}{4C_2^2},\dfrac{C_1}{4}}$ with $C_1,C_2$ defined in Proposition~\ref{prop dk goes 0}.  Since $g$ is convex, we get $\varrho=0,$ and thus
\begin{align*}
C_1=\lambda(1-\lambda L)=\dfrac{\lambda}{2} \text{ and } C_2=2(\sqrt{2\lambda}+\sqrt{2\lambda})=4\sqrt{2\lambda}.
\end{align*}
Therefore, $\mathscr{C}=\dfrac{\lambda}{512}$. The initial radii $\varepsilon_1$ and $r_1$ and the radius reduction factors $\mu,\theta$ are chosen as
\begin{align*}
\varepsilon_1=r_1=\sqrt{\dfrac{100}{\mathscr{C}}}\;\text{ and }\;\mu=\theta =\dfrac{1}{2}.
\end{align*} 
\item iFB: $\omega_k=\dfrac{1}{k^4}.$ This is a typical choice for $\omega_k$ to ensure that $\sum_{k=1}^\infty \sqrt{\omega_k}<\infty.$ 
\end{itemize}

All numerical experiments are conducted on a computer with 10th Gen Intel(R) Core(TM) i5-10400 (6-Core 12M Cache, 2.9GHz to 4.3GHz) and 16GB RAM memory. The codes are written in MATLAB R2021a. The result for the tests is presented in the following tables, where \textit{``iter"} is the number of iterations, and where \textit{``fval"} and \textit{``error"} are the values of the objective function $\varphi$ and the error $\omega_k$ required for the subproblem at the last iteration, respectively.
\begin{center}
\begin{table}[H]
\centering
\small
\begin{tabular}{|l|llllll|lllll|} 
\hline
  & \multicolumn{6}{c|}{IPGM}& \multicolumn{5}{c|}{iFB}  \\ 
\hline
\multicolumn{1}{|c|}{TN} & \multicolumn{1}{c}{iter} & \multicolumn{1}{c}{fval} & \multicolumn{1}{c}{$\norm{g^k}$} & \multicolumn{1}{c}{error} & \multicolumn{1}{c}{$\varepsilon_k$} & \multicolumn{1}{c|}{time} & \multicolumn{1}{c}{iter} & \multicolumn{1}{c}{fval} & \multicolumn{1}{c}{$\norm{g^k}$} & \multicolumn{1}{c}{error} & \multicolumn{1}{c|}{time}  \\ 
\hline
1 & 2014    & 28.59   & 3.2E+01 & 1.5E-06  & 1.0E+01& 3.75     & 2000    & 28.60   & 3.2E+01 & 6.3E-14  & 3.88      \\
2 & 2013    & 78.09   & 9.2E+01 & 6.0E-06  & 3.9E+01& 11.36    & 2000    & 78.10   & 9.2E+01 & 6.3E-14  & 11.46     \\
3 & 2013    & 211.12  & 3.0E+02 & 6.0E-06  & 7.6E+01& 38.47    & 2000    & 211.17  & 3.0E+02 & 6.3E-14  & 42.98     \\
4 & 2012    & 586.02  & 7.8E+02 & 2.4E-05  & 3.0E+02& 340.72   & 2000    & 586.10  & 7.8E+02 & 6.3E-14  & 342\\
5 & 2013    & 237.33  & 2.9E+02 & 6.0E-06  & 7.6E+01& 7.96     & 2000    & 237.37  & 2.9E+02 & 6.3E-14  & 8.18      \\
6 & 2012    & 566.36  & 7.6E+02 & 2.4E-05  & 3.0E+02& 30.3     & 2000    & 566.44  & 7.6E+02 & 6.3E-14  & 31.49     \\
7 & 2014    & 38.91   & 3.3E+01 & 1.5E-06  & 1.0E+01& 24.59    & 2000    & 38.92   & 3.3E+01 & 6.3E-14  & 20.76     \\
8 & 2013    & 97.56   & 9.6E+01 & 6.0E-06  & 3.9E+01& 281.07   & 2000    & 97.58   & 9.6E+01 & 6.3E-14  & 307.96    \\
9 & 2014    & 28.43   & 3.2E+01 & 1.5E-06  & 1.0E+01& 3.8      & 2000    & 28.44   & 3.2E+01 & 6.3E-14  & 4.97      \\
10& 2013    & 77.82   & 9.2E+01 & 6.0E-06  & 3.9E+01& 9.21     & 2000    & 77.84   & 9.2E+01 & 6.3E-14  & 9.18      \\
11& 2013    & 210.71  & 3.0E+02 & 6.0E-06  & 7.6E+01& 36.15    & 2000    & 210.76  & 3.0E+02 & 6.3E-14  & 36.41     \\
12& 2012    & 585.48  & 7.8E+02 & 2.4E-05  & 3.0E+02& 392.7    & 2000    & 585.56  & 7.8E+02 & 6.3E-14  & 388.58    \\
13& 2013    & 237.24  & 2.9E+02 & 6.0E-06  & 7.6E+01& 8.14     & 2000    & 237.28  & 2.9E+02 & 6.3E-14  & 8.06      \\
14& 2012    & 566.23  & 7.6E+02 & 2.4E-05  & 3.0E+02& 30.38    & 2000    & 566.31  & 7.6E+02 & 6.3E-14  & 32.21     \\
15& 2013    & 38.24   & 3.3E+01 & 1.5E-06  & 1.0E+01& 20.1     & 2000    & 38.24   & 3.3E+01 & 6.3E-14  & 20.06     \\
16& 2013    & 96.56   & 9.7E+01 & 6.0E-06  & 3.9E+01& 306.7    & 2000    & 96.58   & 9.7E+01 & 6.3E-14  & 305.78    \\
\hline
\end{tabular}
\caption{Result of the first experiment }
\label{result 1}
\end{table}

\begin{table}[H]
\centering
\small
\begin{tabular}{|l|llllll|lllll|} 
\hline
& \multicolumn{6}{c|}{IPGM}   & \multicolumn{5}{c|}{iFB}\\ 
\hline
\multicolumn{1}{|c|}{TN} & \multicolumn{1}{c}{iter} & \multicolumn{1}{c}{fval} & \multicolumn{1}{c}{$\norm{g^k}$} & \multicolumn{1}{c}{error} & \multicolumn{1}{c}{$\varepsilon_k$} & \multicolumn{1}{c|}{time} & \multicolumn{1}{c}{iter} & \multicolumn{1}{c}{fval} & \multicolumn{1}{c}{$\norm{g^k}$} & \multicolumn{1}{c}{error} & \multicolumn{1}{c|}{time}  \\ 
\hline
1     & 786439& 0.92  & 1.0E-01 & 2.3E-11& 4.0E-02& 1953.72& 10530 & 10.32 & 8.8E+00 & 8.1E-17& 4000    \\
2     & 786866& 2.95  & 6.3E-01 & 3.6E-10& 3.1E-01& 4000   & 9827  & 36.76 & 3.1E+01 & 1.1E-16& 4000    \\
3     & 193820& 24.17 & 1.2E+01 & 2.3E-08& 4.8E+00& 4000   & 7785  & 114.19& 1.1E+02 & 2.7E-16& 4000    \\
4     & 23461 & 227.81& 1.8E+02 & 1.5E-06& 7.5E+01& 4000   & 8200  & 361.19& 3.6E+02 & 2.2E-16& 4000    \\
5     & 921886& 13.67 & 3.0E+00 & 1.5E-09& 1.2E+00& 4000   & 12124 & 119.70& 8.3E+01 & 4.6E-17& 4000    \\
6     & 251253& 73.71 & 3.0E+01 & 2.3E-08& 9.4E+00& 4000   & 10507 & 307.73& 3.1E+02 & 8.2E-17& 4000    \\ CT
7     & 381454& 4.96  & 4.0E-01 & 3.6E-10& 1.6E-01& 4000   & 7835  & 20.81 & 1.5E+01 & 2.7E-16& 4000    \\
8     & 45168 & 26.97 & 1.1E+01 & 9.3E-08& 4.9E+00& 4000   & 6908  & 58.10 & 4.4E+01 & 4.4E-16& 4000    \\
9     & 1127533 & 0.28  & 1.0E-01 & 2.3E-11& 4.0E-02& 2515.67& 50316 & 2.45  & 3.0E+00 & 1.6E-19& 4000    \\
10    & 690078& 1.73  & 7.5E-01 & 3.6E-10& 3.1E-01& 4000   & 43018 & 15.86 & 1.0E+01 & 2.9E-19& 4000    \\
11    & 184279& 23.11 & 1.3E+01 & 2.3E-08& 4.8E+00& 4000   & 39672 & 52.81 & 3.4E+01 & 4.0E-19& 4000    \\
12    & 21417 & 235.58& 1.9E+02 & 1.5E-06& 7.5E+01& 4000   & 23285 & 226.89& 1.8E+02 & 3.4E-18& 4000    \\
13    & 922893& 12.94 & 3.0E+00 & 1.5E-09& 1.2E+00& 4000   & 52519 & 66.71 & 3.3E+01 & 1.3E-19& 4000    \\
14    & 253962& 72.51 & 3.0E+01 & 2.3E-08& 9.4E+00& 4000   & 52220 & 145.44& 9.6E+01 & 1.3E-19& 4000    \\
15    & 389506& 1.82  & 5.5E-01 & 3.6E-10& 1.6E-01& 4000   & 35987 & 8.76  & 3.6E+00 & 6.0E-19& 4000    \\
16    & 38703 & 25.86 & 1.3E+01 & 9.3E-08& 4.9E+00& 4000   & 30827 & 28.76 & 1.4E+01 & 1.1E-18& 4000    \\
\hline
\end{tabular}
\caption{Result of the second experiment}
\label{result 2}
\end{table}
\end{center}\vspace*{-0.2in}
The results in the tests can be summarized as follows:

$\bullet$ Table~\ref{result 1} shows that the performances of the two algorithms are typically the same when the number of iterations is small.

$\bullet$ Table~\ref{result 2} shows that when the large number of iterations is required, IPGM has advantages on the errors control. It can be seen from this table that the error required for subproblems of IPGM is larger than that of iFB. Therefore, IPGM requires less time to solve subproblems than iFB, which makes it faster. Eventually, in all tests in Table~\ref{result 2} except Test~12 and Test~16, iFB stagnates at the number of iteration much smaller than IPGM, and thus IPGM attains smaller objective function than iFB.\vspace*{-0.1in}

\section{ Conclusions and Further Research}\label{sec:concl}

In this paper, we propose and develop inexact proximal point and inexact proximal gradient methods to solve problems of weakly convex optimization and nonconvex composite optimization. These methods achieve stationary accumulation points and, under additional assumptions on the KL property of the envelopes, the global linear convergence. The convergence analysis of the developed algorithms are based on the general framework of finding zeros of continuous mappings when only inexact information for such mappings is accessible. A new approximate subdifferential for weakly convex functions is employed to simplify the convergence analysis of the aforementioned methods.

Our future research includes detailed and efficient numerical analysis and experiments to solve subproblems of the proposed methods. We also intend to find conditions ensuring the convergence to local/global minimizers of our inexact methods. The obtained results would allow us to develop new applications in important classes of models in machine learning, statistic, and related disciplines.

\end{document}